\documentclass[11pt]{article}
\usepackage{amsmath}
\usepackage{amssymb}
\usepackage{titlesec}
\usepackage{graphicx}
\usepackage{fancyhdr}
\usepackage{latexsym}
\usepackage{mathrsfs}
\usepackage{booktabs}
\usepackage{mathrsfs}
\usepackage{mathtools}
\usepackage{amsthm}
\usepackage{wasysym}
\usepackage{cite}
\usepackage{color}
\usepackage{geometry}
\usepackage{hyperref}
\usepackage{cite}
\newcommand{\bnn}{\begin{equation*}}
	\newcommand{\enn}{\end{equation*}}

\newcommand{\be}{\begin{equation}}
	\newcommand{\ee}{\end{equation}}
\newcommand{\ba}{\begin{aligned}}
	\newcommand{\ea}{\end{aligned}}

\hypersetup{hypertex=true,
	colorlinks=true,
	linkcolor=blue,
	anchorcolor=blue,
	citecolor=blue}
\numberwithin{equation}{section}

\newtheorem{theorem}{Theorem}[section]
\newtheorem{lemma}{Lemma}[section]
\newtheorem{remark}{Remark}[section]

\title{On Outer Pressure Problem of   Compressible Navier-Stokes System  with Degenerate Heat-Conductivity   in  Unbounded Domains
}

\author{Manyu Liu$^a$, Yanfang Peng$^b$, Zhilun Peng$^c$
	\thanks{Y. F. Peng was partially supported by  National Nature Science Foundation of China (Grant Nos.12361046).
		Email addresses:  liumy@email.ncu.edu.cn(M. Y. Liu); pyfang2005@sina.com (Y. F. Peng); zhilunp@student.unimelb.edu.au (Z. L. Peng)}\\[3mm]  a. School of Mathematics and Computer Science ,\\
	Nanchang University, Nanchang 330031, P. R. China \\ b. School of Mathematical Sciences, Guizhou Normal University,  \\Guiyang, 550001, P. R. China\\c.  School of Mathematics and Statistics, Faculty of Science,\\ The University of Melbourne, Parkville, VIC 3010, Australia  }

\date{}
\begin{document}
	\maketitle
	\begin{abstract}
		    The  compressible Navier-Stokes system with   the   constant viscosity and the nonlinear heat conductivity which  is proportional to a positive power of the temperature  and may be degenerate is considered. Under the outer pressure boundary conditions in one-dimensional  unbounded spatial domains,  the global existence   of the strong solutions  is obtained after  proving that  both the specific volume and temperature are  bounded from below
		and above independently of time and space. Moreover,  the asymptotically stability of global solutions is established as time tends to infinity.
		
		\noindent
		{\textbf
			{keywords:}} Compressible Navier-Stokes system; Strong solutions; Outer pressure; Degenerate heat conductivity; Large-time Behavior
	\end{abstract}
	\section{Introduction and main results}
	\par \quad \quad
	The compressible Navier-Stokes system describing the one-dimensional motion of a viscous heat-conducting polytorpic gas, can be written in the Lagrange variables as follows (see\cite{a,b}):	
	\begin{equation}
		v_{t}=u_{x}, \label{1.1}
	\end{equation}
	\begin{equation}
		u_{t}+P_{x}=\mu\left(\frac{u_{x}}{v}\right)_{x},  \label{1.2}
	\end{equation}
	\begin{equation}
		\left(e+\frac{u^{2}}{2}\right)_{t}+(P u)_{x}=\left(\kappa \frac{\theta_{x}}{v}+\mu \frac{u u_{x}}{v}\right)_{x},     \label{1.3}
	\end{equation}
	where $t>0$ is time, $x\in \Omega =(0, \infty)\subset \mathbb{R}$ denotes the Lagrange mass coordinate, and the unknown functions $v>0, u, \theta>0, P, e>0$ are the specific volume of the gas, fluid velocity, absolute temperature, pressure and internal energy respectively. In this paper, we focus on the ideal polytropic gas, that is, $P$ and  $e>0$ satisfy
	\begin{equation}
		P=R \theta / v, \, e=c_{v} \theta+\text {const}, \label{1.4}
	\end{equation}
	where $R$ (specific gas constant) and  $c_{v}$ (heat capacity at constant volume) are both  positive constants. Moreover, viscosity coefficient $\mu$ and heat conductivity $\kappa$ are proportional to (possibly different) powers of $\theta$:
	\begin{equation}
		\mu =\tilde{\mu} \theta^\gamma,\,\kappa = \tilde{\kappa} \theta^\beta,
		\label{1.5}
	\end{equation}
	where $\tilde{\mu}$, $\tilde{\kappa}>0$ and $\gamma$, $\beta \ge 0$ are constants.
	
	The system (\ref{1.1})-(\ref{1.5}) is supplemented with the initial condition
	\begin{equation}
		(v(x, 0), u(x, 0), \theta(x, 0))=\left(v_{0}(x), u_{0}(x), \theta_{0}(x)\right), \quad x \in \Omega,
		\label{1.6}
	\end{equation}
	and far-field and boundary ones
	\begin{equation}
		\ \lim _{x \rightarrow\infty}(v, u,\theta)=(1,0,1),
		\label{1.7}
	\end{equation}
	\begin{equation}
		\left(-R\frac{\theta}{v}+\mu\frac{u_x}{v}\right)(0,t)=
		-P(t) \equiv-R,
		\label{1.8}
	\end{equation}
	\begin{equation}
		\theta_{x}(0, t)=0,
		\label{1.9}
	\end{equation}
	where the first boundary condition, \eqref{1.8},  is called outer pressure boundary condition.
	
	There is a  large number of literatures on the large-time existence and behavior of solutions to system (\ref{1.1})-(\ref{1.5}).  For constant coefficients $\gamma=\beta=0$, Kazhikhov and Shelukhin \cite{c} first proved the global existence of solution for large initial data in bounded domains.
	From then on, significant progress has been made on the mathematical aspect of the initial boundary value problems, see \cite{f,g,h,i,j,k,l} and the references therein.
	Motivated by the fact that in the case of isentropic flow, a temperature dependence on the viscosity translates into a density dependence, there are much effort has been made, see \cite{i,m,n,o,p,q} concentrating on the case that $\mu$ is independent of $\theta$, and heat conductivity $\kappa$ is allowed to depend on temperature in a special way with a positive lower bound and balanced with corresponding constitution relations.
	
	When it comes to the case that  $\gamma=0, \beta >0$, the research becomes  more challenging due to the possible degeneracy and strong nonlinearity on viscosity and heat diffusion. On the one hand, for the bounded domain, Jenssen-Karper \cite{r}   showed the global existence of   weak solutions  when $\gamma = 0 $ and $ \beta \in (0,3/2)$. Later on, for $\gamma = 0$ and $\beta >0$, Pan-Zhang \cite{s} (see also \cite{t} and  \cite{u}) obtained the global strong solution. On the other hand, when the domain is unbounded, Li-Shu-Xu \cite{LSX} proved the existence of global strong solutions under the initial conditions that
	$$(v_0-1,u_0,\theta_0-1) \in H^1 \times H^1 \times H^1,  \inf_{x\in \Omega}v_0(x)>0,  \inf_{x\in \Omega}\theta_0(x)>0.$$

	Regarding to the large-time behavior of strong solutions, when $\gamma=0$ and  the domain is bounded, Kazhikhov \cite{kaz1} and Huang-Shi \cite{z} showed that the solution is nonlinearly exponentially stable as time tends to infinity for $\beta=0$ and $\beta>0$ respectively. In the case of unbounded
	domains and large data,  Li-Liang \cite{LL} (see also \cite{bb}, \cite{JS})  established the  large-time
	behavior of strong solutions.  However,  it is worth noticing that  the method
	used in \cite{LL}  relies heavily on the non-degeneracy of the heat conductivity $\kappa$ and cannot
	be applied directly to the degenerate and nonlinear case $\beta>0$. Recently, for the case that $ \beta>0$ and unbounded domain, Li-Xu \cite{LX} made studies  on  the large-time behavior of the global strong solutions obtained in Li-Shu-Xu \cite{LSX}.
	
	In recent years, the outer pressure  problem with $P(t)<0$ in \eqref{1.8} has attracted a lot of  attention. For the bounded domains,  Nagasawa \cite{w} first obtained the convergence of the solutions to a stationary state and the rate of its convergence with $\gamma=\beta = 0$. And then, Cai-Chen-Peng \cite{x} extended their results to $\gamma=0$ and $\beta>0$. Recently, when  the domain is unbounded, for $\gamma=\beta=0$, Han-Wu-Zhang \cite{y} established the existence and asymptotically stability of global strong solutions. It is interesting to study the global existence and large-time behavior of the strong solutions of  the system  (\ref{1.1})-(\ref{1.9}) with $\gamma=0, \beta>0$ in unbounded domains.
We  state our result as follows.	
	\begin{theorem}\label{theorem1}
		Suppose that
		\begin{equation*}
			\gamma = 0, \,\, \beta>0,
		\end{equation*}
		and  the initial data $(v_0,u_0,\theta_0)$ satisfy
		\begin{equation*}
			v_{0}-1, u_{0}, \theta_{0}-1 \in H^{1}(\Omega),\, \inf _{x \in \Omega} v_{0}(x)>0,\,\inf _{x \in \Omega} \theta_{0}(x)>0.
		\end{equation*}
		Then  system (\ref{1.1})-(\ref{1.9}) has a unique global strong solution $(v,u,\theta)$  such that
		\begin{equation}
			\left\{\begin{array}{l}
				v-1, u, \theta-1 \in L^{\infty}\left(0,
				\infty; H^{1}(\Omega)\right),\\
				v_{t} \in L^{\infty}\left(0, \infty; L^{2}(\Omega)\right)\cap L^2(0,\infty;H^1(\Omega)), \\
				u_x,\theta_x, u_t,\theta_t,v_{xt}, u_{xx}, \theta_{xx}\in L^2(0,\infty;L^2(\Omega)).
				\label{1.10}
			\end{array}\right.
		\end{equation}
		Moreover,
		\begin{equation}
			C^{-1}\leq v(x,t)\leq C, \,\, C^{-1}\leq \theta(x,t) \leq C,
			\label{1.11}
		\end{equation}
		\begin{equation}
			\sup_{0\le t <\infty}\|(v-1, u, \theta-1)\|_{H^1(\Omega)} + \int_{0}^{\infty}\left(\|v_{x}\|^2_{L^2(\Omega)}+\|(u_x, \theta_x)\|^2_{H^1(\Omega)}\right)dt \le C,
			\label{1.12}
		\end{equation}
		and for any $p\in(2,\infty]$,
		\begin{equation}
			\lim_{t\rightarrow \infty}\left(\left\|(v-1,u,\theta-1)(t)\right\|_{L^p(\Omega)}
			+\|(v_x, u_x, \theta_x)\|_{L^2(\Omega)}\right)=0,
			\label{1.13}
		\end{equation}		
		where  $C>0$ is a constant depending only on $\tilde{\mu}$, $\tilde{\kappa}$, $\beta$, $R$, $c_{v}$, $\left\|(v_0-1,u_0,\theta_0- 1)\right\|_{H^1(\Omega)}$, $\underset{x \in \Omega}{\inf}v_{0}(x)$ and $\underset{x \in \Omega}{\inf}\theta_{0}(x)$.
	\end{theorem}
	
	\begin{remark}  Theorem \ref{theorem1} can be regarded as a nature generalization result of Han-Wu-Zhang \cite{y} where they discussed the case that $\gamma=\beta=0$.
	\end{remark}
	\begin{remark} Theorem \ref{theorem1} implies that the result is true for the general case that $$P(t)\equiv  P\left(\lim\limits_{x\rightarrow\infty}v(x,t), \lim\limits_{x\rightarrow\infty} \theta(x,t)\right).$$
		Thus, it is an interesting  problem when  $P(t)\not\equiv P\left(\lim\limits_{x\rightarrow\infty}v(x,t),\lim\limits_{x\rightarrow\infty} \theta(x,t)\right)$,  which will be left for future.
	\end{remark}
	
	We now make some comments on the analysis of this paper. The key issue is to get the time-independent lower and upper bounds of $v$ and $\theta$ (see \eqref{2.3}, \eqref{GJ}). Compared with \cite{y} and \cite{LX}, the main difficulties come from the degeneracy and  nonlinearity of the
	heat conductivity due to $\beta>0$ or the outer pressure boundary conditions \eqref{1.8}. Firstly, to arrive at \eqref{2.3}, after  establishing the  standard energetic estimate \eqref{2.1} and  the local expression of $v$ (see \eqref{2.6}),  we make some slight modifications  of  the ideas in \cite{LX,y} to obtain the time-independent lower and upper bounds of $v$ in  the half-space. Next, the key step is  to establish the time-independent  bound on the $L^2(\Omega \times (0,T))$-norm of $u_x$ and $\theta^{-1/2}\theta_x$ (see \eqref{2.28} and \eqref{2.54}). To this end, we  mainly follow some ideas from Li-Liang \cite{LL} and  Li-Xu\cite{LX}. However, due to the outer pressure boundary conditions \eqref{1.8}, some new discussions have to be involved here. Indeed,  on one hand, for $\beta\in (0,1)$, it is essential to control  the integral term
	$$
	\mathbb{B} \triangleq\int_0^\infty \int_{\Omega\setminus\Omega_2(t)}\theta^\beta\theta_x^2dxdt
	$$
	for $\Omega_2(t)\triangleq \{x\in \Omega|\theta(x,t)>2\}$.
	Comparing with the proof of Lemma 2.3 in \cite{LX}, we need to deal with  the additional term (see \eqref{2.31}) caused by the outer pressure boundary condition \eqref{1.8}, that is,
	$$\int_0^T|u(\theta-2)_+(0,t)|dt.
	$$
	By  observing  a simple fact  that $(\theta-3/2)_+\geq 1/2$ when $\theta\geq2$, one can check  the following estimate (see \eqref{2.37})
	\begin{equation*}
		\begin{aligned}
			\int_0^T|u(\theta-2)_+(0,t)|dt
			&\le C\int_0^T\|u_x\|^2_{L^2(\Omega)}dt + C\int_0^T\max_{x\in\Omega}|\theta(x,t)-3/2|^2_+dt.
		\end{aligned}	
	\end{equation*}
	Consequently, we obtain  that $\mathbb{B}$ can be controlled by
	\begin{equation}
		\label{app}
		\int_0^\infty \max_{x\in \Omega}(\theta-3/2)_+^2(x,t)dt,
	\end{equation}
	which is indeed bounded by $C(\varepsilon)+C\varepsilon \mathbb{B}$ for any $\varepsilon>0$ (see \eqref{2.53}). On the other hand, for  $\beta\geq1$, we reach the time-independent bound on the $L^2(\Omega \times (0,T))$-norm of $u_x$ and $\theta^{-1/2}\theta_x$ (see \eqref{2.54}) by a standard argument as \cite{LX}. Finally, to obtain the higher order estimates,  it is crucial to get the bounds on  $L^\infty(0,T;L^2(\Omega))$ and $L^2(0,T; L^2(\Omega))$ of $v_x$. We need to modify some ideas in \cite{y} due to the  degeneracy and nonlinearity of the heat conductivity  $\theta^\beta,\beta>0$. Firstly,  we can treat the outer pressure term  by (\ref{app}) and (see \eqref{2.61})	
	\begin{equation*}
		\begin{aligned}
			\left|\frac{uu_x}{v}(0,t)\right|
			&\le C +C\|u_x\|^2_{L^2(\Omega)} + C\max_{x\in\Omega}|\theta(x,t)-3/2|^2_+,
		\end{aligned}	
	\end{equation*}
	which  gives (see \eqref{aa1})
	\begin{equation*}
		\begin{aligned}
			f_t + \frac{1}{8}f &\le C\int_{{\Omega}}\frac{\theta^2_x}{\theta}dx +C\max_{x\in\Omega}(\theta-3/2)^2_+ +C\|u_x\|^2_{L^2(\Omega)}\\
			& \quad + C\max_{x\in\Omega}(1-\theta)^4_+ f + C,
		\end{aligned}
	\end{equation*}
	where $f(t)$ is defined as \eqref{2.66}. Combining this with Gronwall's inequality, we succeed in getting the desired estimate on the  $L^\infty(0,T;L^2(\Omega))$ of $v_x$.  Then, multiplying (1.2) by $-(u_x-\theta+v)_x$, and modifying  the ideas of \cite{x},  we can derive the  the bound on the $L^2(\Omega \times (0,T))$-norm of $u_{xx}$, $\theta_x$,  $v_x$ and $u_t$ (see Lemmas \ref{lemma q}).  With these estimates at hand, we derive the upper and lower bound of the temperature (see Lemma \ref{lemma2.10}).  Finally, carrying out the standard arguments, we can derive \eqref{1.12} and \eqref{1.13}.

	\section{Proof of Theorem 1.1}
	We start with the local existence result which can be proved by applying the Banach theorem and
	the contraction mapping principle (see\cite{qq1,qq2,qq3}).
	
	\begin{lemma}\label{lemma2.1}
		Under the conditions of Theorem \ref{theorem1}, there exists some $T>0$  such that the initial-boundary-value problem (\ref{1.1})-(\ref{1.9}) has a unique strong solution($v, u, \theta$) satisfying
		\begin{equation*}
			\left\{\begin{array}{l}
				v-1, u, \theta-1 \in L^{\infty}\left(0,
				T; H^{1}(\Omega)\right),\\
				v_{t} \in L^{\infty}\left(0, T; L^{2}(\Omega)\right)\cap L^2(0, T;H^1(\Omega)), \\
				u_x,\theta_x, u_t,\theta_t,v_{xt}, u_{xx}, \theta_{xx}\in L^2(0,T;L^2(\Omega)).
			\end{array}\right.
		\end{equation*}
	\end{lemma}
	Then, after establishing some necessary priori estimates (see  (\ref{2.1}), (\ref{2.28}), (\ref{2.54}), (\ref{2.59}), (\ref{q1}) and (\ref{2.94}) below) where the constants  depend only on the initial data of the problem, we can  extend  the local solution obtained by Lemma \ref{lemma2.1} to the whole interval $[0, \infty)$, which assures the existence of global strong solution in  Theorem \ref{theorem1}.
	
	Without loss of generality, we assume that $\tilde{\mu}=\tilde{\kappa}=R=c_{v}=1$. Throughout
	the paper, $C$ denotes a positive generic constant depending only on $\tilde{\mu}$, $\tilde{\kappa}$, $\beta$, $R$,  $c_{v}$, $\underset{x \in \Omega}{\inf}v_{0}(x)$, $\underset{x \in \Omega}{\inf}\theta_{0}(x)$ and  $\left\|\left(v_{0}-1, u_{0}, \theta_{0}-1\right)\right\|_{H^{1}(\Omega)}$,  which may vary in different cases. And we write $C(\alpha)$ to emphasize that
	$C$ depends on $\alpha$.
	
	We give the following standard energetic estimates.
	\begin{lemma}\label{lemma2.2}
		It holds that
		\begin{equation}
			\begin{aligned}
				&\sup _{0 \leq t<T} \int_{\Omega}\left(\frac{u^2}{2} +(v-\ln v-1)+(\theta-\ln \theta-1)\right) dx
				+ \int_{0}^{T}  V(t) dt\\
				&\quad \quad\leq E_0\triangleq \int_\Omega\left(\frac{u_0^2}{2} +(v_0-\ln v_0-1)+(\theta_0-\ln \theta_0-1)\right)dx,
				\label{2.1}
			\end{aligned}
		\end{equation}
		where
		\begin{equation*}
			V(t)\triangleq \int_{\Omega}\left( \frac{u_{x}^{2}}{v \theta} +  \frac{\theta^\beta \theta_{x}^{2}}{v \theta^{2}} \right)dx.
		\end{equation*}
	\end{lemma}

	\begin{proof}
		Using (\ref{1.1}), (\ref{1.2}) and (\ref{1.4}), we rewrite (\ref{1.3})  as
		
		\begin{equation}
			\theta_{t}+\frac{\theta}{v} u_{x}=\left(\frac{\theta^\beta \theta_{x}}{v}\right)_{x}+ \frac{u_{x}^{2}}{v}.
			\label{2.2}
		\end{equation}
		Multiplying (\ref{1.1}) by $1-v^{-1}$, (\ref{1.2}) by $u$, (\ref{2.2}) by $1-\theta^{-1}$, and adding them altogether, we have
		\begin{equation*}
			\begin{aligned}
				&\left(\frac{u^2}{2}+(v-\ln v-1)+(\theta-\ln \theta-1)\right)_{t}+ \frac{u_{x}^{2}}{v \theta}+ \frac{\theta^\beta \theta_{x}^{2}}{v \theta^{2}} \\
				&=\left(\frac{ u u_{x}}{v}-\frac{ u \theta}{v}\right)_{x}+ u_{x}+\left(\left(1-\theta^{-1}\right) \frac{\theta^\beta \theta_{x}}{v}\right)_{x},
			\end{aligned}
		\end{equation*}
		which combined with  (\ref{1.6})-(\ref{1.9}) gives (\ref{2.1}) and hence finishes the proof of Lemma \ref{lemma2.2} .
	\end{proof}
	
	Now, we get the uniform (with respect to time) upper and lower bounds of $v$.
	\begin{lemma}\label{lemma2.3} There exists a positive constant $C$ such that for any $(x,t)\in \Omega\times [0,\infty)$,
		\begin{equation}
			C^{-1}\leq v(x,t) \leq C.
			\label{2.3}
		\end{equation}	
	\end{lemma}	
	\begin{proof} For any $x\in \Omega$ and $i=[x]$. Setting
		\begin{equation*}
			\sigma(x,t)\triangleq\frac{ u_{x}}{v} - \frac{\theta}{v} = (\ln v)_t - \frac{\theta}{v},
		\end{equation*}
		we rewrite (\ref{1.2}) as
		$$ u_t=\sigma_x.$$
		
		Integrating this with respect to $x$ over $[i,x]$ leads to
		\begin{equation}\label{2.4}
			\left(\int_{i}^{x} u d y\right)_{t}=\sigma(x, t)-\sigma(i, t),
		\end{equation}
		which implies
		\begin{equation*}
			\int_i^x(u(y,t)-u_0(y))dy=\ln v-\ln v_0-\int_0^t\frac{\theta}{v}ds-\int_0^t\sigma(i,s)ds.
		\end{equation*}
		Then,  we have
		\begin{equation}
			v(x,t)=D_i(x, t) Y_i(t) \exp \left\{ \int_{0}^{t} \frac{\theta}{v} d s\right\},
			\label{2.5}
		\end{equation}
		where
		\begin{equation*}
			D_i(x, t)\triangleq v_0(x)\exp\left\{\int_{i}^{x} (u(y,t)-u_{0}) dy\right\},
		\end{equation*}
		and
		\begin{equation}
			Y_i(t)\triangleq \exp \left\{\int_{0}^{t}\sigma(i,s) ds\right\}.
			\label{Y}
		\end{equation}
		
		Setting
		\begin{equation*}
			g\triangleq\int_{0}^{t}  \frac{\theta}{v} ds,
		\end{equation*}
		direct calculation  gives
		\begin{equation*}
			g_{t}= \frac{\theta(x,t)}{v(x,t)}=\frac{ \theta(x,t)}{D_i(x,t)Y_i(t) \exp\{g\}},
		\end{equation*}
		which leads to
		\begin{equation*}
			\exp\{g\}=\int_{0}^{t} \frac{\theta(x,s)}{D_i(x,s)Y_i(s)} d s +1.
		\end{equation*}
		
		Putting this into (\ref{2.5}) yields
		\begin{equation}
			v(x, t)=D_i(x, t) Y_i(t)+\int_{0}^{t} \frac{D_i(x, t) Y_i(t) \theta(x, s)}{D_i(x, s) Y_i(s)}ds.
			\label{2.6}
		\end{equation}
		Now, we estimate $v(x,t)$ as follows.
		
		First, we have by (\ref{2.1})
		\begin{equation*}
			\int_{i}^{i+1}(v-\ln v-1 )+ (\theta-\ln\theta-1)dx\leq E_0,
		\end{equation*}
		which combined  with Jensen's inequality gives
		\begin{equation}
			\alpha_{1} \leq \int_{i}^{i+1} v(x, t) dx\leq \alpha_{2}, \,\,\, \alpha_{1} \leq \int_{i}^{i+1} \theta(x, t)dx \leq \alpha_{2},
			\label{2.7}
		\end{equation}
		where $0<\alpha_1<\alpha_2$ are two roots of
		\begin{equation*}
			y-\ln y -1 = e_0.
		\end{equation*}
		Thus, by (\ref{2.1}) and Cauchy's inequality, we obtain
		\begin{equation*}
			\left|\int^{x}_{i} (u(y,t)-u_0(y)) dy\right|\leq \left(\int_{i}^{i+1} u^2 dy\right)^{\frac{1}{2}}+ \left(\int_{i}^{i+1} u_0^2 dy\right)^{\frac{1}{2}} \leq C,
		\end{equation*}
		which  implies
		\begin{equation}
			C^{-1}\leq D_i(x,t) \leq C,
			\label{2.8}
		\end{equation}
		where  $C$ is a constant independent of $i$ and $T$.
		
		For $0\le s<t\le T$, integrating (\ref{2.4}) over $(i,i+1)\times (s,t)$ leads to
		\begin{equation}
			\begin{aligned}
				\int_{s}^{t}\sigma(i,\tau)d\tau
				&\le \int_{s}^{t} \int_{i}^{i+1}\left(\frac{u_x}{v}
				-\frac{\theta}{v}\right)dxd\tau +C\\
				&\le C\int_{s}^{t}\int_{i}^{i+1}\frac{u^2_x}{v\theta}dxd\tau - \frac{1}{2}\int_{s}^{t}\int_{i}^{i+1} \frac{\theta}{v}dxd\tau +C\\
				&\le C-\frac{1}{2}\int_{s}^{t}\int_{i}^{i+1} \frac{\theta}{v}dxd\tau,
				\label{2.9}
			\end{aligned}
		\end{equation}
		where in the last inequality we have used (\ref{2.1}).
		
		Then, for
		\begin{equation*}
			\bar{\theta}_i(t)\triangleq \int_{i}^{i+1}\theta(x,t)dx,
		\end{equation*}
		we have by (\ref{2.7})
		\begin{equation}
			\begin{aligned}
				|\theta^{1/2}(x,t)-\bar{\theta}_i^{1/2}(t)| &\le C|\theta^\frac{\beta+1}{2}(x,t)-\bar{\theta}_i^\frac{\beta+1}{2}(t)| \\
				&\le C\int_{i}^{i+1}\theta^{\frac{\beta-1}{2}}|\theta_x|dx \\
				&\le C\left(\int_{i}^{i+1}\frac{\theta^\beta\theta_x^2}{\theta^2v}dx\right)^{1/2}\left(\int_{i}^{i+1}\theta vdx\right)^{1/2} \\
				&\le CV^{1/2}(t)\max_{x\in [i,i+1]}v^{1/2}(x,t),
				\label{2.10}
			\end{aligned}
		\end{equation}
		which gives
		\begin{equation}
			\frac{\alpha_1}{4}-CV(t)\max_{x\in [i,i+1]}v(x,t)\le \theta(x,t)\le C+ CV(t)\max_{x\in [i,i+1]}v(x,t).
			\label{2.11}
		\end{equation}
		
		Moreover, it follows from (\ref{2.7}) and Jensen's inequality  that
		\begin{equation*}
			\begin{aligned}
				-\int_{i}^{i+1}\frac{\theta}{v}dx &\le \int_{i}^{i+1}((-\theta + \bar{\theta}_i)_+ - \bar{\theta}_i)\frac{1}{v}dx\\
				&\le \left(\max_{x\in [i,i+1]}(\bar{\theta}_i^{\beta/2}-\theta^{\frac\beta2})_+-\bar{\theta}_i\right)\left(\int_{i}^{i+1}vdx\right)^{-1} \\
				&\le C\int_i^{i+1}\theta^{\frac\beta2-1}|\theta_x|dx-2C^{-1}\\
				&\le CV(t)-C^{-1},
			\end{aligned}
		\end{equation*}
		which together with (\ref{Y}) and (\ref{2.1}) implies
		\begin{equation}
			0 \leq Y_i(t) \leq  C e^{-t/C}, \quad \frac{Y_i(t)}{Y_i(s)} \leq C e^{-(t-s) / C}. 
			\label{2.13}
		\end{equation}	
		Putting (\ref{2.8}), (\ref{2.13}) and  (\ref{2.11}) into (\ref{2.6}) yields
		\begin{equation*}
			\begin{aligned}
				v(x,t)
				&\le Ce^{-t/C} + C\int_{0}^{t}e^{-(t-s)/C}\left(1 + V(s)\max_{x\in [i,i+1]}v(x,s)\right)ds \\
				&\le C + C\int_{0}^{t}V(s)\max_{x\in [i,i+1]}v(x,s)ds,
			\end{aligned}
		\end{equation*}
		which together with Gronwall's inequality and (\ref{2.1}) proves that for any $t \in [0,\infty]$,
		\begin{equation*}
			\max_{x\in [i,i+1]}v(x,t)\le C.
		\end{equation*}
		Since $x$ is arbitrary, we arrive at
		\begin{equation}
			v(x,t) \le C,
			\label{2.14}
		\end{equation}
		for all $(x,t) \in \Omega \times [0,\infty)$.
		
		Then, integrating \eqref{2.6} in $x$ over $(i,i+1)$ and   by (\ref{2.1}), (\ref{2.7})-(\ref{2.8}) and (\ref{2.11})-(\ref{2.13}) yields
		\begin{equation}
			\alpha_1 \le \int_{i}^{i+1}v(x,t)dx \le CY_i(t) + C\int_{0}^{t}\frac{Y_i(t)}{Y_i(s)} ds.
			\label{2.15}
		\end{equation}
	Combining this with   (\ref{2.6})-(\ref{2.8}), (\ref{2.11})-(\ref{2.14}) gives  that there exists  $T_0$ independent of $T$ such that for $t\geq T_0$,
		\begin{equation}
			\begin{aligned}
				v(x,t)
				&\ge C^{-1} \int_{0}^{t}\frac{Y_i(t)}{Y_i(s)}\left(\frac{\alpha_1}{4} - CV(s)\right)ds \\
				&\ge C_0 - Ce^{-t/C} - C\int_{0}^{t}e^{-(t-s)/C}V(s)ds \\
				&\ge \frac{C_0}{2},
				\label{2.16}
			\end{aligned}
		\end{equation}
		where in the last inequality we have used
		\begin{equation*}
			\begin{aligned}			
				\int_{0}^{t}e^{-(t-s)/C}V(s)ds
				&=\int_{0}^{\frac{t}{2}}e^{-(t-s)/C}V(s)ds + \int_{\frac{t}{2}}^{t}e^{-(t-s)/C}V(s)ds\\
				&\le e^{-t/(2C)}\int_{0}^{\frac t2}V(s)ds + \int_{\frac{t}{2}}^{t}V(s)ds \rightarrow 0,\quad \text{as} \,\,\, t\rightarrow \infty.
			\end{aligned}
		\end{equation*}
		
		Finally,  by (\ref{2.15})  and Gronwall's inequality, we have  for all $t\in [0,T_0]$,
		$$ Y_i(t)\ge (CT_0e^{CT_0})^{-1}.$$
		Hence, by (\ref{2.6}) and  (\ref{2.16}),  there exists some positive constant $C$ such that
		$$v(x,t)\ge C^{-1},$$
		for all $(x,t) \in \Omega\times (0,\infty)$. Combining this and (\ref{2.14})  gives (\ref{2.3}) and  the proof of Lemma \ref{lemma2.3} is completed.
	\end{proof}
	
	For further discussion,   we denote for $a\ge 0$,
	$$(\theta>a)(t) \triangleq \{x \in \Omega \mid \theta(x, t)>a\} ,\,\,\, (\theta<a)(t)\triangleq \{x \in \Omega \mid \theta(x, t)< a\}.$$
	When  $\alpha>1$, we derive from (\ref{2.1}) that
	\begin{equation}
		\sup_{0\le t <\infty}\int_{ ({\theta>\alpha})(t)} \theta dx \le C(\alpha)\sup_{0\le t <\infty}\int_{ {\Omega}}(\theta - \ln \theta -1)dx \le C(\alpha),
		\label{2.21}
	\end{equation}
	and that
	\begin{equation}
		\sup_{0\le t <\infty}|(\theta<\alpha^{-1})(t)|+|(\theta>\alpha)(t)|\leq C(\alpha).
		\label{2.22}
	\end{equation}
	
	Similar to the proof of   Lemma 2.2 of \cite{LX}, we have  
	\begin{lemma}\label{lemma2.5}
		For any $p\geq 1$, there exists some positive constant $C(p)$ such that
		\begin{equation}
			\int_{0}^{T}\int_{\Omega}\frac{\theta^\beta \theta^2_x}{\theta^{p+1}}dxdt + \int_{0}^{T}\int_{\Omega}\frac{u^2_x}{\theta^p}dxdt \le C(p).
			\label{2.19}
		\end{equation}
	\end{lemma}

	Next, we state the following key Lemmas \ref{lemma2.6} and \ref{lemma2.7} which will give the $L^2$-norm (in both space and time) bound of both $u_x$ and $\theta^{-1/2}\theta_x$.
	
	\begin{lemma}\label{lemma2.6}
		For $\beta \in (0,1)$, there exists some positive constant C such that for any $T>0$,
		\begin{equation}
			\int_{0}^{T}\int_{{\Omega}}\left(u^2_x + (\theta^{-1} + \theta^\beta)\theta^2_x\right)dxdt\le C.
			\label{2.28}
		\end{equation}
		\begin{proof}
			
			We divide the proof of Lemma \ref{lemma2.6} into three steps.
			
			Step 1. Integrating (\ref{2.2}) multiplied by $(\theta-2)_+\triangleq\max{\{\theta-2,0\}}$ over $\Omega \times (0,T)$ gives
			\begin{equation}
				\begin{aligned}
					&\frac{1}{2}\int_{{\Omega}}(\theta-2)^2_+dx  + \int_{0}^{T}\int_{(\theta>2)(t)}\frac{\theta^\beta \theta^2_x}{v}dxdt \\
					&=  \frac{1}{2}\int_{{\Omega}}(\theta_0-2)^2_+ dx-\int_{0}^{T}\int_{{\Omega}}\frac {\theta}{v}u_x(\theta-2)_+dxdt +\int_{0}^{T}\int_{{\Omega}}\frac{u^2_x}{v}(\theta-2)_+dxdt.
					\label{2.29}
				\end{aligned}
			\end{equation}
			
			To deal with the last term on the right hand side of (\ref{2.29}), motivated by \cite{LL}, we multiply (\ref{1.2}) by $2u(\theta-2)_+$ and integrate the resulting equality over $\Omega \times (0,T)$,
			\begin{equation}
				\begin{aligned}
					&\quad\int_{\Omega} u^{2}(\theta-2)_{+}dx + 2  \int_{0}^{T} \int_{\Omega} \frac{u_{x}^{2}}{v}(\theta-2)_{+} dxdt\\
					&=  \int_{\Omega} u_{0}^{2}\left(\theta_{0}-2\right)_{+}dx + 2 \int_{0}^{T} \int_{\Omega} \frac{\theta}{v} u_{x}(\theta-2)_{+}dxdt+2  \int_{0}^{T} \int_{(\theta>2)(t)} \frac{\theta}{v} u \theta_{x}dxdt \\
					&\quad-2  \int_{0}^{T} \int_{(\theta>2)(t)} \frac{u_{x}}{v} u \theta_{x}dxdt+\int_{0}^{T} \int_{(\theta>2)(t)} u^{2} \theta_{t} dxdt \\
					&\quad- 2\int^{T}_{0} [u(\theta -2)_{+}](0,s)ds . \label{2.30}
				\end{aligned}
			\end{equation}
			Adding (\ref{2.30}) to (\ref{2.29}) and  using (\ref{2.2}) gives that
			\begin{equation}
				\begin{aligned}
					&\int_{\Omega}\left[\frac{1}{2}(\theta-2)_{+}^{2}+u^{2}(\theta-2)_{+}\right]dx   + \int_{0}^{T} \int_{(\theta>2)(t)} \frac{\theta^\beta \theta_{x}^{2}}{v}dxdt +\int_{0}^{T} \int_{\Omega}\frac{u_{x}^{2}}{v}(\theta-2)_+ dxdt \\
					&= \int_{\Omega}\left[\frac{1}{2}\left(\theta_{0}-2\right)_{+}^{2}+u_{0}^{2}\left(\theta_{0}-2\right)_{+}\right]dx + \int_{0}^{T} \int_{\Omega} \frac{\theta}{v} u_{x}(\theta-2)_{+}dxdt \\
					&\quad+ 2\int_{0}^{T} \int_{(\theta>2)(t)} \frac{\theta}{v} u \theta_{x}dxdt -2\int_{0}^{T} \int_{(\theta>2)(t)} \frac{u_{x}}{v} u \theta_{x}dxdt \\
					&\quad+ \int_{0}^{T} \int_{(\theta>2)(t)} u^{2}\left(\frac{u_{x}^{2}}{v}- \frac{\theta}{v} u_{x}\right)dxdt  +\int_{0}^{T} \int_{(\theta>2)(t)} u^{2}\left(\frac{\theta^\beta \theta_{x}}{v}\right)_{x}dxdt\\
					&\quad - 2 \int_{0}^{T}\left[u(\theta-2)_{+}\right](0, t) d t \\
					&\triangleq\int_{\Omega}\left[\frac{1}{2}\left(\theta_{0}-2\right)_{+}^{2}+u_{0}^{2}\left(\theta_{0}-2\right)_{+}\right]dx +  \sum_{i=1}^{6} I_{i}.
					\label{2.31}
				\end{aligned}
			\end{equation}
			
			We estimate each $I_i (i=1,2,\cdots,6)$ as follows.
			
			For $I_1$, we have by  Cauchy's inequality and (\ref{2.3})
			\begin{equation}
				\begin{aligned}
					\left|I_{1}\right| &=\left|\int_{0}^{T} \int_{\Omega} \frac{\theta}{v} u_{x}(\theta-2)_{+}dxdt\right|  \\
					& \leq \frac{1}{2} \int_{0}^{T} \int_{\Omega} \frac{u_{x}^{2}}{v}(\theta-2)_{+}dxdt+C \int_{0}^{T} \int_{\Omega} \theta^{2}(\theta-2)_{+} dxdt\\
					& \leq \frac{1}{2} \int_{0}^{T} \int_{\Omega} \frac{u_{x}^{2}}{v}(\theta-2)_{+}dxdt+C \int_{0}^{T} \int_{\Omega} \theta(\theta-3 / 2)_{+}^{2}dxdt \\
					& \leq \frac{1}{2} \int_{0}^{T} \int_{\Omega} \frac{u_{x}^{2}}{v}(\theta-2)_{+}dxdt+C \int_{0}^{T} \max _{x \in\Omega}(\theta-3 / 2)_{+}^{2}(x, t) dt,
					\label{2.32}
				\end{aligned}
			\end{equation}
			where in the last inequality we have used (\ref{2.21}).
			
			Now, it follows from Cauchy's inequality that for any $\varepsilon>0$,
			\begin{equation}
				\begin{aligned}
					\left|I_{2}\right|+\left|I_{3}\right| &=2 \left|\int_{0}^{T} \int_{(\theta>2)(t)} \frac{\theta}{v}u\theta_{x}dxdt\right| + 2 \left|\int_{0}^{T} \int_{(\theta>2)(t)} \frac{u_{x}}{v}u\theta_{x}dxdt\right| \\
					& \le C\int_{0}^{T} \int_{\Omega} \theta^2_xdxdt + C\int_{0}^{T}\int_{(\theta>2)(t)} u^{2} \theta^{2}dxdt +  C\int_{0}^{T}\int_{\Omega} u^{2} u_{x}^{2} dxdt \\
					&\le C(\varepsilon)\int_{0}^{T}\int_{{\Omega}}\theta^{\beta-2}\theta_x^2dxdt + \varepsilon\int_{0}^{T}\int_{\Omega}\theta^\beta\theta^2_xdxdt + C\int_{0}^{T}\int_{(\theta>2)(t)} u^{2} \theta^{2}dxdt\\
					&\quad + C\int_{0}^{T}\int_{{\Omega}}u^2_xdxdt +  C\int_{0}^{T}\int_{\Omega}\left|u\right|^{2/(1-\beta)}u^2_xdxdt\\
					&\le C(\varepsilon) + \varepsilon\int_{0}^{T}\int_{\Omega}\theta^\beta\theta^2_xdxdt + C\int_{0}^{T}\sup_{x \in\Omega}(\theta-3/2)^2_+(x,t)dt \\
					&\quad + C\int_{0}^{T}\int_{{\Omega}}u^2_xdxdt + C\int_{0}^{T}\int_{\Omega}\left|u\right|^{2/(1-\beta)}u^2_xdxdt,
					\label{2.33}
				\end{aligned}
			\end{equation}
			where in the last inequality we have used
			
			\begin{equation}
				\begin{aligned}
					\int_{0}^{T}\int_{(\theta>2)(t)}u^2\theta^2dxdt &\le 16\int_{0}^{T}\int_{{\Omega}}u^2(\theta-3/2)^2_+dxdt \\
					&\le 16\int_{0}^{T}\sup_{x \in\Omega}(\theta-3/2)^2_+\int_{{\Omega}}u^2dxdt\\
					&\le C\int_{0}^{T}\sup_{x \in\Omega}(\theta-3/2)^2_+(x,t)dt,
					\label{2.34}
				\end{aligned}
			\end{equation}
			due to (\ref{2.1}).
			
			Then, combining  Cauchy's inequality and (\ref{2.34}) yields that
			\begin{equation}
				\begin{aligned}
					\left|I_{4}\right| &= \int_{0}^{T} \int_{(\theta>2)(t)} u^{2}\left(\frac{u_{x}^{2}}{v}- \frac{\theta}{v} u_{x}\right)dxdt\\
					&\le C\int_{0}^{T}\int_{{(\theta>2)}(t)}u^2u_x^2dxdt + C\int_{0}^{T}\int_{{(\theta>2)}(t)}u^2\theta^2dxdt\\
					&\le C\int_{0}^{T}\int_{{\Omega}}u^2_xdxdt + C\int_{0}^{T}\int_{\Omega}\left|u\right|^{2/(1-\beta)}u^2_xdxdt\\
					&\quad + C\int_{0}^{T}\sup_{x \in\Omega}(\theta-3/2)^2_+(x,t)dt .
					\label{2.35}
				\end{aligned}
			\end{equation}
			
			To estimate $I_5$, we set  $\eta>0$ and
			\begin{equation*}
				\varphi_{\eta}(\theta) := \begin{cases}1, & \theta-2>\eta, \\ (\theta-2) / \eta, & 0 \leq \theta-2 \leq \eta, \\ 0, & \theta-2 \leq 0.
				\end{cases}
			\end{equation*}
			By Lebesgue's dominated convergence theorem, for $\beta<1$ and any $\varepsilon>0$, we have 
			\begin{equation}
				\begin{aligned}
					I_{5} &= \lim _{\eta \rightarrow 0+} \int_{0}^{T} \int_{\Omega}  \varphi_{\eta}(\theta) u^{2}\left(\frac{\theta^\beta\theta_{x}}{v}\right)_{x} dxdt\\
					&= \lim _{\eta \rightarrow 0+} \int_{0}^{T} \int_{\Omega}\left(-2 \varphi_{\eta}(\theta) u u_{x} \frac{\theta^\beta \theta_{x}}{v}-\varphi_{\eta}^{\prime}(\theta) u^{2} \frac{\theta^\beta \theta_{x}^{2}}{v}\right)dxdt \\
					& \leq-2 \int_{0}^{T} \int_{(\theta>2)(t)} u u_{x} \frac{\theta^\beta \theta_{x}}{v}dxdt \\
					&\le C(\varepsilon)\int_{0}^{T}\int_{{(\theta>2)}(t)}u^2u_x^2\theta^\beta dxdt + \varepsilon\int_{0}^{T}\int_{{(\theta>2)}(t)}\theta^\beta \theta_x^2dxdt\\
					&\le C(\varepsilon) \int_{0}^{T}\int_{{\Omega}}\left|u\right|^{2/(1-\beta)}u^2_xdxdt + \varepsilon\int_{0}^{T}\int_{{\Omega}}(u_x^2\theta + \theta^\beta \theta_x^2)dxdt,
				\end{aligned} \label{2.36}
			\end{equation}
			where in the third inequality we have used $\varphi_\eta'(\theta)\geq 0$.	
			
			Finally,  Cauchy's inequality and (\ref{2.1}) lead to
			\begin{equation}
				\begin{aligned}
					\left|I_{6}\right| &
					=2\left|\int_0^T\left[u(\theta-2)_+\right](0,t)dt\right|\\
					&\leq C\int_{0}^{T} \max_{x\in\Omega} \left|u  (\theta-3/2)^{3/2}_{+}\right| d t \\
					& \leq C\int_{0}^{T}  \max_{x\in\Omega}   u^4d t+C \int_{0}^{T}\max_{x\in\Omega} (\theta-3/2)^{2}_{+}  d t \\
					& \leq C \int_{0}^{T} \int_{\Omega} u^2_{x}  d x d t +C  \int_{0}^{T} \max _{x \in \Omega}(\theta-3 / 2)_{+}^{2} d t    ,
				\end{aligned}
				\label{2.37}
			\end{equation}
			where in the  last inequality we have used the following  fact that
			\begin{equation}
				\begin{aligned}
					\max_{x\in\Omega}u^2(x,t)&\le \max_{x \in \Omega}\left(-2\int_{x}^{\infty}u(y)u_x(y)dy\right)\\
					&\le 2\|u\|_{L^2(\Omega)}\left\|u_{x}\right\|_{L^{2}(\Omega)}\\&\le C\|u_x\|_{L^2(\Omega)},
				\end{aligned}
				\label{2.38}
			\end{equation}
			due to (\ref{2.1}).	
			
			Notice that
			\begin{equation*}
				\begin{aligned}
					&\int_{0}^{T}\int_{\Omega}\left(u^2_x\theta + \theta^\beta\theta^2_x\right)dxdt \\
					&\le \int_{0}^{T}\int_{{(\theta>3)(t)}}\left(u^2_x\theta + \theta^\beta\theta^2_x\right)dxdt + \int_{0}^{T}\int_{{(\theta<4)(t)}}\left(u^2_x\theta + \theta^\beta\theta^2_x\right)dxdt \\
					&\le C\int_{0}^{T}\int_{{(\theta>3)(t)}}\left(u^2_x(\theta-2)_+ +\theta^\beta\theta^2_x\right)dxdt + C\int_{0}^{T}\int_{{(\theta<4)(t)}}\left(\frac{u_x^2}{\theta}+ \frac{\theta^\beta\theta^2_x}{\theta^2}\right)dxdt \\
					&\le C\int_{0}^{T}\int_{(\theta>2)(t)}\frac{1}{v}\left(u^2_x(\theta-2)_+ + \theta^\beta\theta^2_x\right)dxdt + C,
				\end{aligned}
			\end{equation*}
			where in the third inequality we have used (\ref{2.3}) and (\ref{2.1}).
			
			Putting (\ref{2.32}), (\ref{2.33}), (\ref{2.35})-(\ref{2.37}) into (\ref{2.31}) and choosing $\varepsilon$ suitably small, we  have
			\begin{equation}
				\begin{aligned}
					& \sup _{0 \leq t \leq T} \int_{\Omega}(\theta-2)_{+}^{2}dx+\int_{0}^{T} \int_{\Omega}\left(u_{x}^{2} \theta+\theta^\beta\theta_{x}^{2}\right)dxdt \\
					&\leq  C+ C \int_{0}^{T} \max _{x \in \Omega}(\theta-3 / 2)_{+}^{2}(x, t)dt + C\int_{0}^{T}\int_{{\Omega}}u_x^2dxdt\\
					& \quad +C_{2} \int_{0}^{T} \int_{\Omega}\left|u\right|^{2/({1-\beta)}}u^2_xdxdt\\
					&\leq C +  C \int_{0}^{T} \max _{x \in \Omega}(\theta-3 / 2)_{+}^{2}(x, t)dt +C_{2} \int_{0}^{T} \int_{\Omega}\left|u\right|^{2/({1-\beta)}}u^2_xdxdt   ,
					\label{2.40}
				\end{aligned}
			\end{equation}
			where  in the last inequality we have used the following  fact that for any $\delta>0$,
			\begin{equation}
				\begin{aligned}
					\int_{0}^{T}\int_{{\Omega}}u^2_xdxdt &\le \delta\int_{0}^{T}\int_{{\Omega}}\theta u^2_xdxdt +C(\delta)\int_{0}^{T}\int_{{\Omega}}\frac{u^2_x}{\theta}dxdt \\
					&\le \delta\int_{0}^{T}\int_{{\Omega}}\theta u^2_xdxdt+C(\delta),
					\label{2.41}
				\end{aligned}
			\end{equation}
			due to Cauchy's inequality, (\ref{2.1}) and (\ref{2.3}).
			
			Step 2. To estimate the last term on the right hand side of (\ref{2.40}), we rewrite (\ref{1.2}) as
			\begin{equation}
				u_t - \left(\frac{u_x-\theta+ v}{v}\right)_x = 0.
				\label{a1}
			\end{equation}
			Multiplying (\ref{a1}) by $|u|^\alpha u(\alpha = 2/(1-\beta))$ and integrating the resulting equality over $\Omega \times (0,T)$, we  get
			\begin{equation}
				\begin{aligned}
					&\frac{1}{\alpha+2}\int_{{\Omega}}|u|^{\alpha+2}dx + (\alpha+1)\int_{0}^{T}\int_{\Omega}\frac{|u|^{\alpha} u^2_x}{v}dxdt \\
					& \le \frac{1}{\alpha+2}\int_{{\Omega}}|u_0|^{\alpha+2}dx + \int_{0}^{T}\int_{{\Omega}}\left(\frac{u_x- \theta + v }{v}|u|^{\alpha}|u_x|\right)_xdxdt\\
					& \quad+ (\alpha+1)\int_{0}^{T}\int_{{\Omega}}\left(\frac{|1-v|}{v} +  \frac{|\theta-1|}{v}\right)|u|^{\alpha}|u_x|dxdt\\
					&\le \frac{1}{\alpha+2}\int_{{\Omega}}|u_0|^{\alpha+2}dx + C\int_{0}^{T}\int_{{\Omega}}\frac{|1-v|}{v} |u|^{\alpha}|u_x|dxdt \\
					&\quad +C\int_{0}^{T}\int_{{(\theta<3)(t)}}\frac{|\theta-1|}{v}|u|^{\alpha}|u_x|dxdt +  C\int_{0}^{T}\int_{(\theta>2)(t)}\frac{|\theta-1|}{v}|u|^\alpha |u_x|dxdt \\
					&\triangleq C + \sum_{i=1}^{3} J_{i},
					\label{2.42}
				\end{aligned}
			\end{equation}
			where in the second inequality we have used  (\ref{1.7}) and (\ref{1.8}).
			
			By (\ref{2.1}) and (\ref{2.3}), we deduce that for any $\delta \in [2,4]$,
			\begin{equation}
				\begin{aligned}
					&\sup_{0\leq t\leq T }\int_\Omega (v-1)^2dx+\sup_{0\leq t\leq T }\int_{(\theta<\delta)(t) } (\theta-1)^2dx\\
					&\leq C\sup_{0\leq t\leq T }\int_\Omega (v-\ln v-1)dx+C \sup_{0\leq t\leq T }\int_\Omega (\theta-\ln\theta-1)dx\leq C,
					\label{2.43}
				\end{aligned}
			\end{equation}
			which together with Holder's inequality yields that
			\begin{equation}
				\begin{aligned}
					|J_1|+|J_2|& \le C\int_{0}^{T}\||u|^\alpha u_x\|_{L^2(\Omega)}\left(\|v-1\|_{L^{2}(\Omega)}+ \|\theta-1\|_{L^{2}((\theta<3)(t))}\right)dt\\
					& \le C\int_{0}^{T}\||u|^\alpha u_x\|_{L^2(\Omega)}dt.
					\label{2.44}
				\end{aligned}
			\end{equation}
			
			Next, if $\beta\le 1/2$,
			\begin{equation}
				\begin{aligned}
					\||u|^\alpha u_x\|_{L^2(\Omega)} & \le C\max_{x\in\Omega}u^2(x,t)\||u|^{\alpha\beta}u_x\|_{L^2(\Omega)}  \\
					& \le C(\varepsilon)\int_{{\Omega}}u^2_xdx + \varepsilon\int_{{\Omega}}|u|^\alpha u^2_xdx,
					\label{2.45}
				\end{aligned}
			\end{equation}
			where in the second inequality we have used (\ref{2.1}) and (\ref{2.38}).
			
			If $\beta \in (1/2,1)$, we have
			\begin{equation}
				\begin{aligned}
					\||u|^\alpha u_x\|_{L^2(\Omega)} &\le C(\varepsilon)\max_{x\in\Omega}|u|^\alpha(x,t) + \varepsilon\int_{{\Omega}}|u|^\alpha u^2_xdx \\
					&\le C(\varepsilon)\int_{{\Omega}}u^2_xdx + 2\varepsilon\int_{{\Omega}}|u|^\alpha u^2_xdx,
					\label{2.47}
				\end{aligned}
			\end{equation}
			where we have used
			\begin{equation*}
				\begin{aligned}
					\max_{x\in\Omega}|u|^\alpha &= \max_{x\in\Omega}\left(-\int_{x}^{\infty}(|u|^\alpha)_x dx\right) \\
					&\le C\int_{{\Omega}}|u_x||u|^{\alpha-1}dx \\
					&\le C\left(\int_\Omega |u|^{\alpha-4}u^2_xdx\right)^{1/2} \left(\int_{{\Omega}}|u|^{\alpha+2}dx\right)^{1/2} \\
					&\le C\left(\int_{{\Omega}}|u|^{\alpha-4}u^2_xdx\right)^{1/2}\max_{x\in\Omega}|u|^{\alpha/2},
				\end{aligned}
			\end{equation*}
			due to $\alpha = {2/(1-\beta)}>4$. Thus, combining (\ref{2.44}), (\ref{2.45}) and (\ref{2.47}) implies for $\beta \in (0,1)$,
			\begin{equation}
				|J_1|+|J_2|\le C(\varepsilon)\int_{0}^{T}\int_{{\Omega}}u^2_xdxdt + C\varepsilon\int_{0}^{T}\int_{{\Omega}}|u|^\alpha u^2_xdxdt.
				\label{2.48}
			\end{equation}
			
			For $J_3$, it follows from Cauchy's inequality and (\ref{2.34}) that
			\begin{equation}
				\begin{aligned}
					|J_3|&\le C(\varepsilon)\int_{0}^{T}\int_{(\theta>2)(t)}(\theta-1)^2|u|^{\alpha}dxdt + \varepsilon\int_{0}^{T}\int_{{\Omega}}\frac{|u|^\alpha u^2_x}{v}dxdt \\
					&\le C(\varepsilon)\int_{0}^{T}\max_{x\in\Omega}(\theta - 7/4)_+^{\frac{2\alpha}{\alpha+2}}\int_{\Omega}\left((\theta-1)^2 + |u|^{\alpha+2}\right)dxdt \\
					&\quad +\varepsilon \int_{0}^{T}\int_{{\Omega}}\frac{|u|^\alpha u^2_x}{v}dxdt \\
					& \le C(\varepsilon)\int_{0}^{T}\max_{x\in\Omega}((\theta(x,t)-3/2)^{\beta+2}_+\theta^{-1})\int_\Omega\left((\theta-1)^2+|u|^{\alpha+2}\right)dxdt \\
					&\quad + \varepsilon\int_{0}^{T}\int_{{\Omega}}\frac{|u|^\alpha u^2_x}{v}dxdt,
					\label{2.49}
				\end{aligned}
			\end{equation}
			due to $\beta \in (0,1)$ and $2\alpha/(\alpha+2)<\beta+1$.
			
			Standard computation yields that
			\begin{equation*}
				\begin{aligned}
					&\max_{x\in\Omega}((\theta(x,t)-3/2)^{\beta+2}_+\theta^{-1}) \\
					&=\max_{x\in\Omega}\left(-\int_{x}^{\infty}((\theta - 3/2)^{\beta+2}_+\theta^{-1})_xdx\right) \\
					&\le C\int_{\Omega}|\theta_x|(\theta-3/2)^{\beta+1}_+\theta^{-1}dx + C\int_{\Omega}|\theta_x|(\theta-3/2)^{\beta+2}_+\theta^{-2}dx \\
					&\le C\left(\int_{{\Omega}}\theta^{\beta-2}\theta^2_xdx\right)^{1/2}\left(\int_{{\Omega}}(\theta-3/2)_+^{\beta+2}dx\right)^{1/2}
					\\
					&\le C \left(\int_{{\Omega}}\theta^{\beta-2}\theta^2_xdx\right)^{1/2}\max_{x\in\Omega}\left((\theta(x,t)-3/2)^{\beta+2}_+\theta^{-1}\right)^{1/2}
					\left(\int_{(\theta>3/2)(t)} \theta dx\right)^{1/2}\\
					&\le C \left(\int_{{\Omega}}\theta^{\beta-2}\theta^2_xdx\right)^{1/2}\max_{x\in\Omega}\left((\theta(x,t)-3/2)^{\beta+2}_+\theta^{-1}\right)^{1/2},
				\end{aligned}
			\end{equation*}
			where in the last inequality we have used (\ref{2.21}). This implies
			\begin{equation}
				\max_{x\in\Omega}((\theta(x,t)-3/2)^{\beta+2}_+\theta^{-1}) \le C\int_{{\Omega}}\theta^{\beta-2} \theta_x^2dx.
				\label{2.50}
			\end{equation}
			
			Then, substituting (\ref{2.48}), (\ref{2.49}) and (\ref{2.50}) into (\ref{2.42}) and choosing $\varepsilon$ suitable small lead to
			\begin{equation}
				\begin{aligned}
					&\sup_{0\le t <T}\int_{{\Omega}}|u|^{\alpha+2}dx + \int_{0}^{T}\int_{{\Omega}}|u|^\alpha u^2_xdxdt \\
					&\le C\int_{0}^{T}\int_{{\Omega}}\theta^{\beta-2}\theta^2_xdx\int_{{\Omega}}((\theta-1)^2 + |u|^{\alpha+2})dxdt + C\int_{0}^{T}\int_{{\Omega}}u^2_xdxdt +C \\
					&\le C\int_{0}^{T}\int_{{\Omega}}\theta^{\beta-2}\theta^2_xdx\int_{\Omega}((\theta-1)^2 + |u|^{\alpha+2})dxdt + C(\delta) \\
					&\quad + C\delta\int_{0}^{T}\int_{{\Omega}}\theta u^2_xdxdt,
					\label{2.51}
				\end{aligned}
			\end{equation}
			where in the last inequality we have used (\ref{2.41}).
			
			Adding (\ref{2.51}) multiplied by $C_2+1$ to (\ref{2.40}), we get by  choosing $\delta$ suitable small
			\begin{equation}
				\begin{aligned}
					&\sup_{0 \leq t<T}\int_{{\Omega}}\left[(\theta-1)^2 + |u|^{\alpha+2}\right]dx + \int_{0}^{T}\int_{{\Omega}}\left[(\theta + |u|^\alpha) u^2_x + \theta^\beta \theta^2_x \right] dxdt \\
					&\le C+C\int_{0}^{T}\left(\int_{{\Omega}}\frac{\theta^\beta\theta^2_x}{v\theta^2}dx\int_{{\Omega}}[(\theta-1)^2 + |u|^{\alpha+2}]dx\right)dt \\
					&\quad + C\int_{0}^{T}\sup_{x \in \Omega}(\theta-3/2)^2_+(x,t)dt,
					\label{2.52}
				\end{aligned}
			\end{equation}
			where we have used
			\begin{equation*}
				\begin{aligned}
					\int_{{\Omega}}(\theta-1)^2dx&\le C\int_{(\theta>3)(t)}(\theta-1)^2dx + \int_{(\theta<4)(t)}(\theta-1)^2dx \\
					&\le C\int_{{\Omega}}(\theta-2)^2_+dx + C,
				\end{aligned}
			\end{equation*}
			due to (\ref{2.43}).
			
			Step 3. It remains to estimate the last term on the right hand side of (\ref{2.52}). In  fact, we have  for $\delta \ge-1$,
			\begin{equation*}
				\begin{aligned}
					&\max_{x \in \Omega}(\theta(x,t)-3/2)^{\delta+3}_+ \\
					&= \max_{x \in \Omega}\left(-\int_{x}^{\infty}((\theta-3/2)^{\delta+3}_+)_xdx\right) \\
					&\le C\int_{{\Omega}}|\theta_x|(\theta-3/2)^{\delta+2}_+dx \\
					&\le C\left(\int_{(\theta>3/2)(t)}\theta^2_x\theta^\delta dx\right)^{1/2}\left(\int_{(\theta>3/2)(t)}(\theta(x,t) - 3/2)^{2\delta+4}_+\theta^{-\delta}dx\right)^{1/2} \\
					&\le C\left(\int_{(\theta>3/2)(t)}\theta^2_x\theta^\delta dx\right)^{1/2}\max_{x \in \Omega}(\theta(x,t)-3/2)^{(\delta+3)/2}_+,
				\end{aligned}
			\end{equation*}
			where in the last inequality we have used (\ref{2.21}). This shows
			\begin{equation}
				\max_{x \in \Omega}(\theta(x,t)-3/2)^{\delta+3}_+ \le C\int_{(\theta>3/2)(t)}\theta^2_x\theta^\delta dx.
				\label{2.57}
			\end{equation}
			Choosing $\delta = -1$ in (\ref{2.57}) and using $\beta<1$, we have
			\begin{equation}
				\begin{aligned}
					\int_0^T\max_{x \in \Omega}(\theta(x,t)-3/2)_+^2dt
					&\le C\int_0^T\int_\Omega \theta^{-1}\theta_x^2dxdt\\
					&\le C\int_0^T\int_\Omega \theta^{\beta-2}\theta_x^2dxdt+\varepsilon\int_0^T\int_\Omega \theta^{\beta}\theta_x^2dxdt.
				\end{aligned}
				\label{2.53}
			\end{equation}
			Putting this into (\ref{2.52}), choosing $\varepsilon$ suitable small, and using Gronwall's inequality lead to
			\begin{equation*}
				\sup_{0 \leq t<T}\int_{{\Omega}}[(\theta-1)^2 + |u|^{\alpha+2}]dx + \int_{0}^{T}\int_{{\Omega}}\left[(\theta + |u|^\alpha)u^2_x + \theta^\beta \theta^2_x\right]dxdt\le C,
			\end{equation*}
			which combined with (\ref{2.41}) and (\ref{2.1}) immediately gives (\ref{2.28}). The proof of Lemma \ref{lemma2.6} is completed.
		\end{proof}
	\end{lemma}
	
	\begin{lemma}\label{lemma2.7}
		For $\beta \in [1,\infty)$, there exists some positive constant C such that for any $T>0$,
		\begin{equation}
			\int_{0}^{T}\int_{{\Omega}}(u^2_x + \theta^{-1}\theta^2_x)dxdt\le C.
			\label{2.54}
		\end{equation}
		\begin{proof} The proof of Lemma \ref{lemma2.7} is simliar to that of   \cite[Lemma 2.4]{LX}. For reader's convenience, we state the proof as follows.
			
			Choosing $p=\beta$ in (\ref{2.19}), we have
			\begin{equation*}
				\int_{0}^{T}\int_{{\Omega}}\theta^{-1}\theta^2_xdxdt\le C.
			\end{equation*}
			Thus, we only need  to prove
			\begin{equation}
				\int_{0}^{T}\int_{{\Omega}}u^2_xdxdt\le C.
				\label{2.55}
			\end{equation}
			Indeed, integrating (\ref{2.2}) multiplied by $(\theta-2)_+\theta^{-1}$  over $\Omega \times (0,T)$ leads to
			\begin{equation}
				\begin{aligned}
					&\int_{0}^{T}\int_{{\Omega}}\frac{u^2_x}{v}(\theta-2)_+\theta^{-1}dxdt \\
					& =-\int_{0}^{T}\int_{{(\theta>2)}(t)}\frac{\theta^\beta\theta_x}{v}(1-2\theta^{-1})_xdxdt+ \int_{0}^{T}\int_{{\Omega}}\frac{(\theta-2)_+}{v} u_xdxdt\\
					&\quad+ \int_{{\Omega}}\int_{\theta_0}^{\theta}(s-2)_+s^{-1}dsdx\\ &=2\int_{0}^{T}\int_{(\theta>2)(t)}\frac{\theta^{\beta}\theta^2_x}{v\theta^2}dxdt  + \int_{0}^{T}\int_{{\Omega}}\frac{(\theta-2)_+}{v} u_xdxdt \\
					&\quad + \int_{{\Omega}}\int_{2}^{\theta}(s-2)_+s^{-1}dsdx -\int_{{\Omega}}\int_{2}^{\theta_0}(s-2)_+ s^{-1}dsdx  \\
					& \le C + C(\varepsilon)\int_{0}^{T}\max_{x \in \Omega}(\theta-3/2)^{\beta+1}_+(x,t)dt + \varepsilon\int_{0}^{T}\int_{{\Omega}}u^2_xdxdt.
					\label{2.56}
				\end{aligned}
			\end{equation}
			Since $\beta\ge 1$, choosing $\delta = \beta-2$ in (\ref{2.57}) gives
			\begin{equation}
				\int_{0}^{T}\max_{x \in \Omega}(\theta(x,t)-3/2)^{\beta+1}_+dt \le C.
				\label{2.58}
			\end{equation}					
			Finally, we have by (\ref{2.1}), (\ref{2.56}) and (\ref{2.58}),
			\begin{equation*}
				\begin{aligned}
					\int_{0}^{T}\int_{{\Omega}}u^2_xdxdt &\le C\int_{0}^{T}\int_{(\theta>3)(t)}\frac{u^2_x}{v}(\theta - 2)_+\theta^{-1}dxdt  + C\int_{0}^{T}\int_{(\theta<4)(t)}\frac{u^2_x}{v\theta}dxdt \\
					&\le C(\varepsilon) + C\varepsilon\int_{0}^{T}\int_{{\Omega}}u^2_xdxdt,
				\end{aligned}
			\end{equation*}
			which directly gives (\ref{2.55}) and finish the proof of Lemma \ref{lemma2.7}.
		\end{proof}
	\end{lemma}			
	For further discussion,  we  have the following time-independent estimate  on the $L^\infty(0,T; L^2(\Omega))$-norm  of $v_x$
	which is crucial in our analysis.

	\begin{lemma}
		\label{lemma2.8}
		There exists some positive constant $C$ such that for any $T>0$.
		\begin{equation}
			\sup_{0 \leq t \leq T}\int_{{\Omega}}v^2_xdx\le C,
			\label{2.59}
		\end{equation}
	\end{lemma}
	\begin{proof} By \eqref{1.1}, we rewrite \eqref{1.2} as
		$$
		u_t+\left(\frac{\theta}{v}\right)_x=\left(\frac{v_x}{v}\right)_t.
		$$
		Integrating this multiplied by $\frac{v_x}{v}$ over $\Omega$, we have by (\ref{1.1}) and Cauchy inequality
		\begin{equation}
			\begin{aligned}
				&\frac{1}{2}\frac{d}{dt}\int_{{\Omega}}\frac{v^2_x}{v^2}dx +\int_{{\Omega}}\frac{\theta v^2_x}{v^3}dx \\
				&= \int_{{\Omega}}\frac{\theta_x v_x}{v^2}dx -\frac{uu_x}{v}(0,t)+\frac{d}{dt}\int_{{\Omega}}u\frac{v_x}{v}dx+ \int_{\Omega}\frac{u_x^2}{v}dx \\
				&\le C\int_{\Omega}\frac{\theta^2_x}{\theta}dx + \frac{1}{4}\int_{{\Omega}}\frac{\theta v^2_x}{v^3}dx +\left|\frac{uu_x}{v}(0,t)\right|+\frac{d}{dt}\int_{{\Omega}}u\frac{v_x}{v}dx+ \int_{\Omega}\frac{u_x^2}{v}dx\\
				&\le C\int_{\Omega}\frac{\theta^2_x}{\theta}dx + \frac{1}{4}\int_{{\Omega}}\frac{(1-\theta) v^2_x}{v^3}dx + \frac{1}{2}\int_{{\Omega}}\frac{\theta v^2_x}{v^3}dx-\frac{1}{4} \int_\Omega \frac{v_x^2}{v^3}dx\\
				&\quad+\left|\frac{uu_x}{v}(0,t)\right|+\frac{d}{dt}\int_{{\Omega}}u\frac{v_x}{v}dx+ \int_{\Omega}\frac{u_x^2}{v}dx  \\
				&\le C\int_{\Omega}\frac{\theta^2_x}{\theta}dx -\frac{1}{8}\int_{{\Omega}}\frac{v^2_x}{v^3}dx + C\max_{x\in\Omega}(1-\theta)^4_+\int_{{\Omega}}\frac{v^2_x}{v^2}dx+\frac{1}{2}\int_{{\Omega}}\frac{\theta v^2_x}{v^3}dx \\
				&\quad+  \left|\frac{uu_x}{v}(0,t)\right| +  \frac{d}{dt}\int_{{\Omega}}u\frac{v_x}{v}dx + \int_{\Omega}\frac{u^2_x}{v}dx.
				\label{2.60}
			\end{aligned}
		\end{equation}
		
		Then, it follows from (\ref{1.7}) that for any $\varepsilon>0$,
		\begin{equation}
			\begin{aligned}
				\left|\frac{uu_x}{v}(0,t)\right| &= \left|\frac{u(\theta-v)}{v}(0,t)\right| \\
				&\le (C|u||\theta|+|u|)(0,t) \\
				&\le C(|u||\theta-3/2|_+ + |u|) (0,t) \\
				&\le C(\|u_x\|^{\frac{1}{2}}_{L^2(\Omega)} \max_{x\in\Omega}|\theta(x,t)-3/2|_+ + C\|u_x\|^{\frac{1}{2}}_{L^2(\Omega)}) \\
				&\le C +C\|u_x\|^2_{L^2(\Omega)} + C\max_{x\in\Omega}|\theta(x,t)-3/2|^2_+,
				\label{2.61}
			\end{aligned}	
		\end{equation}
		where in the third inequality we have used (\ref{2.38}). Putting (\ref{2.61}) into (\ref{2.60}) leads to
		\begin{equation}
			\begin{aligned}
				&\frac{1}{2}\frac{d}{dt}\int_{{\Omega}}\frac{v^2_x}{v^2}dx + \frac{1}{8}\int_{{\Omega}}\frac{v^2_x}{v^2}dx + \frac12\int_{{\Omega}}\frac{\theta v^2_x}{v^2}dx \\
				&\quad \le C\int_{{\Omega}}\frac{\theta^2_x}{\theta}dx + C\max_{x\in\Omega}(1-\theta)^4_+\int_{{\Omega}}\frac{v^2_x}{v^2}dx +  \frac{d}{dt}\int_{{\Omega}}u\frac{v_x}{v}dx \\
				&\quad \quad + C\max_{x\in\Omega}(\theta(x,t)-3/2)^2_+ + C\|u_x\|^2_{L^2(\Omega)}  + C.
				\label{2.62}
			\end{aligned}
		\end{equation}
		Combining (\ref{2.57}) with $\delta = -1$, (\ref{2.28}) and (\ref{2.54}), we obtain
		\begin{equation}
			\int_{0}^{T}\max_{x\in\Omega}(\theta - 3/2)^2_+dt\le C.
			\label{2.63}
		\end{equation}
		
		Moreover, we infer  from (\ref{2.54}), (\ref{2.28}) and (\ref{2.43}) that
		\begin{equation}
			\begin{aligned}
				\int_{0}^{T}\max_{x\in\Omega}(1-\theta)^4_+dt & = 4\int_{0}^{T}\max_{x\in\Omega}\left(\int_{x}^{\infty}(1-\theta)_+\theta_xdx\right)^2dt \\
				&\le C\int_{0}^{T}\left(\int_{{\Omega}}(1-\theta)_+\theta^{-1/2}|\theta_x|dx\right)^2dt \\
				&\le C\int_{0}^{T}\int_{{\Omega}}(1-\theta)^2_+dx\int_{{\Omega}}\theta^{-1}\theta^2_xdxdt\\
				&\le C.
				\label{2.64}
			\end{aligned}
		\end{equation}

		Finally, by (\ref{2.1}), there exists some positive constant $\widetilde{C}$ such that
		\begin{equation}
			f(t)\triangleq \int_{{\Omega}}\frac{v^2_x}{v^2}dx - 2\int_{{\Omega}}u\frac{v_x}{v}dx + \widetilde{C} \ge \frac{1}{2}\int_{{\Omega}}\frac{v^2_x}{v^2}dx,
			\label{2.66}
		\end{equation}
		which together with $(\ref{2.62})-(\ref{2.66}), (\ref{2.3})$ gives
		\begin{equation}
			\begin{aligned}
				f_t + \frac{1}{8}f &\le C\int_{{\Omega}}\frac{\theta^2_x}{\theta}dx +C\max_{x\in\Omega}(\theta-3/2)^2_+ +C\|u_x\|^2_{L^2(\Omega)}\\
				& \quad + C\max_{x\in\Omega}(1-\theta)^4_+f + C.
				\label{aa1}
			\end{aligned}
		\end{equation}
		Combining this with Gronwall's inequality,  (\ref{2.28}), (\ref{2.54}), (\ref{2.58}), (\ref{2.64}) gives (\ref{2.59}),   and hence the proof of Lemma \ref{lemma2.8} is completed.
	\end{proof}
	Now, we can derive the following  necessary uniform-in-time estimates.
	\begin{lemma}
		\label{lemma q}
		There exist a positive constant C such that for any $T>0$
		\begin{equation}
			\sup_{0 \leq t \leq T}\int_{{\Omega}}u_x^2dx + \int_{0}^{T}\int_{{\Omega}}(u_{xx}^2+\theta_x^2+v_x^2+u_t^2)dxdt\le C.
			\label{q1}
		\end{equation}
	\end{lemma}
	\begin{proof}
		First, multiplying (\ref{2.2}) by $(\theta-2)^{\beta+1}_+$ and integrating by parts yields
		\begin{equation}
			\begin{aligned}
				&\frac{1}{\beta+2}\frac{d}{dt}\left(\int_{{\Omega}}(\theta-2)^{\beta+2}_+dx\right) + (\beta +1)\int_{{\Omega}}\frac{\theta^\beta (\theta -2)^\beta_+ \theta_x^2}{v}dx \\
				&= \int_{{\Omega}}\frac{(\theta-2)^{\beta+1}_+u_x^2}{v}dx - \int_{{\Omega}}\frac{\theta(\theta -2)^{\beta +1}_+u_x}{v}dx \\
				&\triangleq K_1 + K_2.
				\label{2.71}
			\end{aligned}
		\end{equation}
		A straightforward calculation indicates that
		\begin{equation}
			\begin{aligned}
				K_1 &\le C\max_{x \in \Omega}(\theta-2)^{\beta+1}_+\int_{{\Omega}}u_x^2dx \\
				& \le C\int_{{\Omega}}(\theta -2)^\beta_+ |\theta_x|dx\int_{{\Omega}}u_x^2dx\\
				& \le C\left(\int_{{\Omega}}\theta^\frac{\beta}{2}(\theta-2)^\frac{\beta}{2}_+|\theta_x|dx\right)\int_{{\Omega}}u^2_xdx \\
				&\le \varepsilon\int_{{\Omega}}\frac{\theta^\beta (\theta-2)^\beta_+\theta^2_x}{v}dx + C(\varepsilon)\left(\int_{{\Omega}}u_x^2dx\right)^2,
				\label{2.72}
			\end{aligned}
		\end{equation}
		and that
		\begin{equation}
			\begin{aligned}
				K_2 &\le C\int_{{\Omega}}\theta(\theta-2)^{\beta+1}_+|u_x|dx\\
				&\le C\max_{x \in \Omega}(\theta-2)_+^{\frac{\beta}{2}+1}\int_{{\Omega}}\theta(\theta-2)^\frac{\beta}{2}_+|u_x|dx \\
				& \le C\left(\int_{{\Omega}}(\theta-2)^\frac{\beta}{2}_+|\theta_x|dx\right)\int_{{(\theta>2)(t)}}\left((\theta-2)^{\frac{\beta}{2}+1}|u_x|+|u_x|\right)dx\\
				&\le C\left(\int_{(\theta>2)(t)}(\theta-2)^\beta_+\theta^2_xdx\right)^\frac{1}{2}
				\left[\left(\int_{{\Omega}}(\theta-2)^{\beta+2}_+dx\right)^\frac{1}{2} \left(\int_{{\Omega}}u_x^2dx\right)^\frac{1}{2} + \left(\int_{{\Omega}}u_x^2dx\right)^\frac{1}{2}\right]\\
				& \le \varepsilon\int_{{(\theta>2)(t)}}\frac{\theta^\beta(\theta-2)^\beta_+\theta^2_x}{v}dx + C(\varepsilon)\left(\int_{{\Omega}}(\theta-2)^{\beta+2}_+dx\right)\int_{{\Omega}}u_x^2dx + C(\varepsilon)\int_{{\Omega}}u_x^2dx.
				\label{2.73}
			\end{aligned}
		\end{equation}
		By applying Gronwall's inequality to (\ref{2.71}) after substituting (\ref{2.72}) and (\ref{2.73}), we obtain
		\begin{equation}
			\begin{aligned}
				\int_{{\Omega}}(\theta-2)^{\beta+2}_+dx + \int_{0}^{T}\int_{{\Omega}}\frac{\theta^\beta(\theta-2)^\beta_+\theta_x^2}{v}dx dt\le C+ C\int_{0}^{T}\left(\int_{{\Omega}}u_x^2dx\right)^2dt.
				\label{2.74}
			\end{aligned}
		\end{equation}
		After using (\ref{2.22}) and (\ref{2.74}), we obtain that
		\begin{equation}
			\begin{aligned}
				\int_{0}^{T}\int_{{\Omega}}\theta^{2\beta}\theta_x^2dxdt &\le \int_{0}^{T}\int_{(\theta>3)(t)}\theta^{2\beta}\theta_x^2dxdt + \int_{0}^{T}\int_{(\theta<4)(t)}\theta^{2\beta}\theta_x^2dxdt\\
				&\le C\int_{0}^{T}\int_{(\theta>3)(t)}\theta^\beta(\theta-2)^\beta_+\theta_x^2dxdt + C\int_{0}^{T}\int_{(\theta<4)(t)}\theta^{-1}\theta_x^2dxdt\\
				& \le  C\int_{0}^{T}\int_{{\Omega}}\frac{\theta^\beta(\theta-2)^\beta_+\theta_x^2}{v}dxdt +C \\
				&\le C+ C\int_{0}^{T}\left(\int_{{\Omega}}u_x^2dx\right)^2dt,
				\label{2.75}
			\end{aligned}
		\end{equation}
		and therefore
		\begin{equation}
			\begin{aligned}
				\int_{0}^{T}\int_{{\Omega}}\theta_x^2dxdt &= \int_{0}^{T}\int_{{\Omega}}\theta_x^21_{(\theta>2)}dxdt + \int_{0}^{T}\int_{{\Omega}}\theta_x^21_{(\theta<2)}dxdt\\
				&\le C\int_{0}^{T}\int_{{\Omega}}\theta^{2\beta}\theta_x^21_{(\theta>2)}dxdt + C\int_{0}^{T}\int_{{\Omega}}\frac{\theta_x^2}{\theta}1_{(\theta<2)}dxdt\\
				&\le C+ C\int_{0}^{T}\left(\int_{{\Omega}}u_x^2dx\right)^2dt.
				\label{q4}
			\end{aligned}
		\end{equation}
		
		Next, we rewrite (\ref{1.2}) as
		\begin{equation}
			u_t-\frac{(u_x-\theta)_x}{v} - \frac{v_x}{v^2} = \Phi,
			\label{q6}
		\end{equation}
		with $\Phi:=- v^{-2}(u_x-(\theta-1))v_x$ satisfying
		\begin{equation}
			\begin{aligned}
				\int_{{\Omega}}\Phi^2dx&\le C\int_{{\Omega}}v_x^2u_x^2dx + C\int_{{\Omega}}v_x^2(\theta-1)^2dx\\
				&\le C\max_{x \in \Omega}u_x^2\int_{{\Omega}}v_x^2dx + C\max_{x \in \Omega}(\theta-1)^2\int_{{\Omega}}v_x^2dx\\
				& \le C\|u_x\|_{L^2}\|u_{xx}\|_{L^2} + C\|v_x\|_{L^2}\|\theta-1\|_{L^2}\|\theta_x\|_{L^2}\\
				& \le \varepsilon(\|u_{xx}\|_{L^2}^2+\|v_x\|_{L^2}^2) + C(\varepsilon)\left(\|u_x\|_{L^2}^2 + \|\theta_x\|_{L^2}^2\right),
				\label{q7}
			\end{aligned}
		\end{equation}
		where in the third inequality we used $\eqref{2.59}$ and the facts that
		\begin{equation}
			\max_{x\in\Omega}u_x^2\le C\|u_x\|_{L^2(\Omega)}\|u_{xx}\|_{L^2(\Omega)}.
			\label{q8}
		\end{equation}
		
		Then, multiplying (\ref{q6}) by $-(u_x-\theta+v)_x$ and integrating by parts gives
		\begin{equation}
			\begin{aligned}
				&\frac{1}{2}\frac{d}{dt}\int_{\Omega}u_x^2dx + \int_{{\Omega}}\frac{vu_{xx}^2+v\theta_{x}^2+v_x^2}{v^2}dx +\int_{{\Omega}}\frac{(u_x-\theta)_xv_x(1+v)}{v^2}dx \\
				&=\left(\int_{{\Omega}}u_x(\theta-v)dx\right)_t + \int_{{\Omega}}\left(\frac{2u_{xx}\theta_{x}}{v}+u_x(v-\theta)_t\right)dx - \int_{{\Omega}}(u_x-\theta+v)_x\Phi dx.
				\label{q10}
			\end{aligned}
		\end{equation}
		
		For the third term on the left hand of \eqref{q10},   we have by (\ref{1.1})
		\begin{equation}
			\begin{aligned}
				&\int_{{\Omega}}\frac{(u_x-\theta)_xv_x(1+v)}{v^2}dx\\
				&=\int_{{\Omega}}\frac{v_{xt}v_x(1+v)}{v^2}dx - \int_{{\Omega}}\frac{\theta_xv_x(1+v)}{v^2}dx\\
				& = \frac{1}{2}\left(\int_{{\Omega}}\frac{v_x^2(1+v)}{v^2}dx\right)_t + \frac{1}{2}\int_{{\Omega}}\frac{v_x^2u_x(v+2)}{v^3}dx - \int_{{\Omega}}\frac{\theta_xv_x(1+v)}{v^2}dx\\
				&\ge \frac{1}{2}\left(\int_{{\Omega}}\frac{v_x^2(1+v)}{v^2}dx\right)_t - C(\varepsilon)\int_{{\Omega}}v_x^2u_x^2dx - \varepsilon\int_{{\Omega}}v_x^2dx - C(\varepsilon)\int_{{\Omega}}\theta_x^2dx\\
				&\ge \frac{1}{2}\left(\int_{{\Omega}}\frac{v_x^2(1+v)}{v^2}dx\right)_t - C\varepsilon\int_{{\Omega}}(v_x^2+ u_{xx}^2)dx - C(\varepsilon)\int_{{\Omega}}(u_x^2 + \theta_x^2)dx,
				\label{q11}
			\end{aligned}
		\end{equation}
		where in the last inequality we have used (\ref{q8}).
		
		For the second term on the right hand of (\ref{q10}), we use
		(\ref{2.2}) and (\ref{1.9}) to deduce
		\begin{equation}
			\begin{aligned}
				&\int_{{\Omega}}\left(\frac{2u_{xx}\theta_{x}}{v}+u_x(v-\theta)_t\right)dx \\&= \int_{{\Omega}}\frac{2u_{xx}\theta_x}{v}dx+\int_{{\Omega}}u_x^2dx- \int_{{\Omega}}\frac{(u_x-\theta)u_x^2}{v}dx+\int_\Omega\frac{\theta^\beta\theta_xu_{xx}}{v}dx\\
				&\le C\int_{{\Omega}}u_x^2dx +C \max_{x\in \Omega}|u_x|\int_{{\Omega}}u_x^2dx + C\max_{x \in \Omega}|\theta-1|\int_{{\Omega}}u_x^2dx\\
				&\quad + \varepsilon\int_{{\Omega}}u_{xx}^2dx + C(\varepsilon)\int_{{\Omega}}\theta_x^2dx+ C(\varepsilon)\int_{{\Omega}}\theta^{2\beta}\theta_x^2dx\\
				&\le C(\varepsilon)\int_{{\Omega}}u_x^2dx + C(\varepsilon)\left(\int_{{\Omega}}u_x^2dx\right)^2 + \varepsilon\int_{{\Omega}}u_{xx}^2dx + C(\varepsilon)\int_{{\Omega}}\theta_x^2dx+ C(\varepsilon)\int_{{\Omega}}\theta^{2\beta}\theta_x^2dx,
				\label{q12}
			\end{aligned}
		\end{equation}
		where in the last inequality we have used the following facts:
		\begin{equation*}
			\begin{aligned}
				\max_{x\in \Omega}|u_x|\int_{{\Omega}}u_x^2dx&\le \|u_x\|_{L^2}^{5/2}\|u_{xx}\|_{L^2}^{1/2}\\
				&\le C(\varepsilon)\|u_x\|_{L^2}^{10/3} + \varepsilon\|u_{xx}\|_{L^2}^2\\
				&\le C(\varepsilon)(\|u_x\|_{L^2}^2 + \|u_x\|_{L^2}^4) + \varepsilon\|u_{xx}\|_{L^2}^2,
			\end{aligned}
		\end{equation*}
		and
		\begin{equation*}
			\begin{aligned}
				\max_{x \in \Omega}|\theta-1|\int_{{\Omega}}u_x^2dx& \le \|\theta-1\|_{L^2}^{1/2}\|\theta_x\|_{L^2}^{1/2}\|u_x\|_{L^2}^2\\
				&\le C\|\theta_x\|_{L^2}^{1/2}\|u_x\|_{L^2}^2\\
				&\le C(\|\theta_x\|^2_{L^2}+\|u_x\|_{L^2}^{8/3})\\
				&\le C(\|\theta_x\|^2_{L^2}+\|u_x\|_{L^2}^2 +\|u_x\|_{L^2}^4 ).
			\end{aligned}
		\end{equation*}
		
		For the last term on the right hand of (\ref{q10}), combining Cauchy's inequality with (\ref{q7}) yields that
		\begin{equation}
			\begin{aligned}
				&\left|\int_{{\Omega}}\left((u_x-\theta)_x+ v_x\right)\Phi dx\right|\\
				& \le \frac{1}{2}\int_{{\Omega}}\frac{vu_{xx}^2+v\theta_{x}^2+v_x^2}{v^2}dx + C\int_{{\Omega}}\Phi^2dx\\
				&\le \frac{1}{2}\int_{{\Omega}}\frac{vu_{xx}^2+v\theta_{x}^2+v_x^2}{v^2}dx +  C\varepsilon\int_{{\Omega}}(v_x^2+ u_{xx}^2)dx + C(\varepsilon)\int_{{\Omega}}(u_x^2+\theta_x^2)dx.
				\label{q13}
			\end{aligned}
		\end{equation}
		Putting (\ref{q11})-(\ref{q13}) into (\ref{q10}), choosing $\varepsilon$ small enough and making use of $\eqref{2.3}, \eqref{2.28}, \eqref{2.54}, \eqref{2.75}$, $\eqref{q4}$, we obtain that after integrating the result over $(0,T)$
		\begin{equation}
			\begin{aligned}\label{zz}
				\sup_{0\le t\le T}\int_\Omega u_x^2dx+ \int_{0}^{T}\int_\Omega (u_{xx}^2+\theta_x^2+v_x^2)dxdt
				\le C\int_{0}^{T}\left(\int_\Omega u_x^2dx\right)^2dt+C,
			\end{aligned}
		\end{equation}
		and where we have used the following estimate
		\begin{equation*}
			\begin{aligned}
				\left|\int_{{\Omega}}u_x(\theta-v)dx\right|& \le \frac{1}{4}\int_{{\Omega}}u_x^2dx + C\int_{{\Omega}}\left((\theta-1)^2+(v-1)^2\right)dx\\
				&\le \frac{1}{4}\int_{{\Omega}}u_x^2dx + C.
			\end{aligned}
		\end{equation*}
		Then, applying Gronwall's equality to $\eqref{zz}$ and using  $\eqref{2.28}, \eqref{2.54}$ gives
		\begin{equation}
			\sup_{0 \leq t \leq T}\int_{{\Omega}}u_x^2dx + \int_{0}^{T}\int_{{\Omega}}(u_{xx}^2+\theta_x^2+v_x^2)dxdt\le C.
			\label{q14}
		\end{equation}
		
		Finally, we write (\ref{1.2}) as
		\begin{equation*}
			u_t = \frac{u_{xx}}{v} - \frac{u_xv_x}{v^2} - \frac{\theta_x}{v} + \frac{\theta v_x}{v^2},
		\end{equation*}
		which together with (\ref{q14}), (\ref{2.28}) and (\ref{2.54}) that
		\begin{equation*}
			\int_{0}^{T}\int_{{\Omega}}u_t^2dxdt\le C.
		\end{equation*}
		Combing this with (\ref{q14}) finishes the proof of Lemma \ref{lemma q}.	
	\end{proof}

	Now, we can prove the uniform lower and upper bounds of the temperature $\theta$.
	
	\begin{lemma}
		\label{lemma2.10}
		There exists a positive constant C such that for any $(x,t) \in \Omega \times[0, T]$
		\begin{equation}
			C^{-1}\le \theta(x,t) \le C.
			\label{GJ}
		\end{equation}
		\begin{proof}
			Multiplying (\ref{2.2}) by $ \theta^ \beta \theta_t $ and integrating the resultant equality over $\Omega$ yields
			\begin{equation*}
				\begin{aligned}
					&\int_{{\Omega}}\theta^\beta\theta_t^2 dx +\int_{{\Omega}}\frac{\theta^{\beta+1}\theta_tu_x}{v}dx \\
					&= \int_{{\Omega}} \left(\frac{\theta^\beta\theta_x}{v}\right)_x\theta^\beta \theta_tdx + \int_{{\Omega}}\frac{\theta^\beta\theta_tu_x^2}{v}dx\\
					&=- \int_{{\Omega}}\frac{\theta^\beta\theta_x}{v}\left(\theta^\beta\theta_t\right)_xdx + \int_{{\Omega}}\frac{\theta^\beta\theta_tu_x^2}{v}dx \\
					& = - \int_{{\Omega}}\frac{\theta^\beta\theta_x}{v}\left(\theta^\beta\theta_x\right)_tdx + \int_{{\Omega}}\frac{\theta^\beta\theta_tu_x^2}{v}dx \\
					& = -\frac{1}{2}\int_{{\Omega}}\frac{\left((\theta^\beta\theta_x)^2\right)_t}{v}dx+ \int_{{\Omega}}\frac{\theta^\beta\theta_tu_x^2}{v}dx \\
					& = -\frac{1}{2}\left(\int_{{\Omega}}\frac{(\theta^\beta\theta_x)^2}{v}dx\right)_t - \frac{1}{2}\int_{{\Omega}}\frac{(\theta^\beta\theta_x)^2u_x}{v^2}dx + \int_{{\Omega}}\frac{\theta^\beta\theta_tu_x^2}{v}dx,
				\end{aligned}
			\end{equation*}
			which gives
			\begin{equation}
				\begin{aligned}
					&\frac{1}{2}\left(\int_{{\Omega}}\frac{(\theta^\beta\theta_x)^2}{v}dx\right)_t + \int_{{\Omega}}\theta^\beta\theta_t^2dx  \\
					&=-\frac{1}{2}\int_{{\Omega}}\frac{(\theta^\beta\theta_x)^2u_x}{v^2}dx - \int_{{\Omega}}\frac{\theta^{\beta+1}\theta_tu_x}{v}dx + \int_{{\Omega}}\frac{\theta^\beta\theta_tu_x^2}{v}dx \\
					&\le C\max_{x\in\Omega}|u_x|\int_{{\Omega}}(\theta^\beta\theta_x)^2dx + \frac{1}{2}\int_{\Omega}\theta^\beta\theta_t^2dx + C\int_{{\Omega}}\theta^{\beta+2}u_x^2dx \\
					&\quad+ C\int_{{\Omega}}(\theta-2)^\beta_+u_x^4dx + C\int_{{\Omega}}u_x^4dx\\
					&\le C\left(\int_{{\Omega}}(\theta^\beta\theta_x)^2dx\right)^2 + C\max_{x\in\Omega}u_x^2 + C\max_{x \in \Omega}u_x^4 +  C\int_{{\Omega}}u_x^2dx + \frac{1}{2}\int_{{\Omega}}\theta^\beta\theta_t^2dx,
					\label{2.87}
				\end{aligned}
			\end{equation}
			due to (\ref{2.74}) and (\ref{q1}).
			
			Then, it follows from (\ref{q1}) and  (\ref{q8}) that
			\begin{equation}
				\int_{0}^{T}\max_{x\in\Omega}(u_x^2 + u_x^4)dt\le C,
				\label{2.88}
			\end{equation}	
			which together with (\ref{2.87}),  Gronwall's inequality and (\ref{2.75}) leads to
			\begin{equation}
				\sup_{0\le t <T}\int_{{\Omega}}(\theta^\beta\theta_x)^2dx + \int_{0}^{T}\int_{{\Omega}}\theta^\beta\theta_t^2dxdt \le C.
				\label{2.89}
			\end{equation}
			Combining this with (\ref{2.22}) in particular gives
			\begin{equation*}
				\begin{aligned}
					\max_{x\in\Omega}(\theta-2)_+ &\le\int_{(\theta>2)(t)}|\theta_x|dx\\
					&\le C\left(\int_{(\theta>2)(t)}(\theta^\beta\theta_x)^2dx\right)^{1/2} \le C,
				\end{aligned}
			\end{equation*}
			which implies that for all $(x,t) \in \Omega \times [0,\infty)$,
			\begin{equation}
				\theta(x,t)\le C.
				\label{2.90}
			\end{equation}
			
			Next, multiplying (\ref{2.2}) by $(\theta-1)^5$ and integrating the resulting equality over $\Omega$ yields
			\begin{equation}
				\begin{aligned}
					&\frac{1}{6}\frac{d}{dt}\int_{{\Omega}}(\theta-1)^6dx \\
					&\le \left|-5\int_{{\Omega}}\frac{(\theta-1)^4\theta^{\beta}\theta_x^2}{v}dx - \int_{{\Omega}}\frac{\theta(\theta-1)^5u_x}{v}dx + \int_{{\Omega}}\frac{(\theta-1)^5u_x^2}{v}dx\right|\\
					& \le C\int_{{\Omega}}\theta_x^2dx + C\int_{{\Omega}}(\theta-1)^6dx + C\int_{{\Omega}}u_x^2dx \\
					& \le C \int_{{\Omega}}\theta_x^2dx + C\int_{{\Omega}}u_x^2dx,
					\label{2.91}
				\end{aligned}
			\end{equation}
			where in the last inequality we have used
			\begin{equation*}
				\int_{{\Omega}}(\theta-1)^6dx\le C\left(\int_{{\Omega}}(\theta-1)^2dx\right)^2\int_{{\Omega}}\theta_x^2dx \le C\int_{{\Omega}}\theta_x^2dx,
			\end{equation*}	
			due to (\ref{2.1}) and (\ref{2.90}). Combining this, (\ref{2.91}) and (\ref{q1}) gives
			\begin{equation}
				\lim_{t \rightarrow \infty}\int_{{\Omega}}(\theta-1)^6dx=0.
				\label{2.92}
			\end{equation}
			Moreover, Sobolev's inequality shows
			\begin{equation*}
				\begin{aligned}
					\max_{x\in\Omega}(\theta^{\beta+1}-1)^2 &\le C\left(\int_{{\Omega}}(\theta^{\beta+1}-1)^6dx\right)^{1/4}\left(\int_{{\Omega}}\theta^{2\beta}\theta_x^2dx\right)^{1/4}\\
					& \le C\left(\int_{{\Omega}}(\theta-1)^6dx\right)^{1/4},
				\end{aligned}
			\end{equation*}
			which together with (\ref{2.92}) implies that there exists some $T_0>0$ such that
			\begin{equation}
				\theta(x,t) \ge 1/2,
				\label{2.93}
			\end{equation}
			for all $(x,t)\in \Omega \times[T_0,\infty)$. Finally, it follows from Lemma 2.3 in \cite{LSX} that there exists some constant $C\ge2$ such that
			\begin{equation*}
				\theta(x,t)\ge C^{-1},
			\end{equation*}	
			for all $(x,t)\in \Omega \times[0,T_0]$. Combining this, (\ref{2.93}) and  (\ref{2.90})  finishes the proof of Lemma \ref{lemma2.10}.
		\end{proof}
	\end{lemma}
	
	Finally, we have the following uniform estimate on the $L^2(\Omega\times(0,T))$-norm of $\theta_t$ and $\theta_{xx}$.
	
	\begin{lemma}
		\label{lemma 2.11}
		There exists a positive constant C such that
		\begin{equation}
			\sup_{0 \leq t \leq T}\int_{{\Omega}}\theta_x^2dx + \int_{0}^{T}\int_{{\Omega}}(\theta_t^2+\theta_{xx}^2)dxdt \le C.
			\label{2.94}
		\end{equation}
		Moreover, for any $p\in(2,\infty]$,
		\begin{equation}
			\lim_{t\rightarrow \infty}\left(\left\|(v-1,u,\theta-1)(t)\right\|_{L^p(\Omega)}
			+\|(v_x, u_x, \theta_x)\|_{L^2(\Omega)}\right)=0.
			\label{w1}
		\end{equation}		
	\end{lemma}
	\begin{proof}
		First, both \eqref{GJ} and (\ref{2.89}) lead to
		\begin{equation}
			\sup_{0 \leq t \leq T}\int_{{\Omega}}\theta_x^2dx + \int_{0}^{T}\int_{{\Omega}}\theta_t^2dxdt \le C.
			\label{2.95}
		\end{equation}

		Then, multiplying (\ref{2.2}) by $\theta_{xx}$ and integrating the result over $\Omega $, we have
		\begin{equation}
			\begin{aligned}	\frac{1}{2}\frac{d}{dt}\int_{ {\Omega}}\theta_x^2dx + \int_{ {\Omega}}\frac{\theta^{\beta}\theta_{xx}^2}{v}dxdt= \int_{ {\Omega}}\left(\frac{\theta^{\beta}\theta_xv_x}{v^2}-\frac{\beta\theta^{\beta-1}\theta_x^2}{v}-\frac{u_x^2}{v} +\frac{\theta u_x}{v}\right)\theta_{xx}dx,
				\label{qwe}
			\end{aligned}
		\end{equation}
		for the right hand term, using Cauchy's inequality gives
		\begin{equation}
			\begin{aligned}
				&\int_{0}^{T}\int_{ {\Omega}}\left|\frac{\theta^{\beta}\theta_xv_x}{v^2}-\frac{\beta\theta^{\beta-1}\theta_x^2}{v}-\frac{u_x^2}{v} +\frac{\theta u_x}{v}\right| \left|\theta_{xx}\right|dxdt\\
				& \le   \varepsilon \int_{0}^{T}\int_{ {\Omega}}\frac{\theta^{\beta}\theta_{xx}^2}{v}dxdt + C(\varepsilon)\int_{0}^{T}\int_{ {\Omega}}\left(\theta_x^2 v_x^2+  \theta_x^4 + u_x^4 + \theta^2u_x^2 \right)dxdt  \\
				& \le \varepsilon \int_{0}^{T}\int_{ {\Omega}}\frac{\theta^{\beta}\theta_{xx}^2}{v}dxdt +  C(\varepsilon)\int_{0}^{T}\max_{x \in \Omega}\theta_x^2\int_{ {\Omega}}v_x^2dxdt +C(\varepsilon) \int_{0}^{T}\max_{x \in \Omega}\theta_x^2\int_{ {\Omega}}\theta_x^2dxdt  \\
				&\quad + C(\varepsilon) \int_{0}^{T}\max_{x \in \Omega}u_x^2\int_{ {\Omega}}u_x^2dxdt +  C(\varepsilon)\int_{0}^{T}\int_{ {\Omega}}u_x^2dxdt  \\
				&\le  \varepsilon \int_{0}^{T}\int_{ {\Omega}}\frac{\theta^{\beta}\theta_{xx}^2}{v}dxdt + C ,
				\label{1234}
			\end{aligned}
		\end{equation}
		where in the last inequality we have used (\ref{2.59}), (\ref{q1}), (\ref{2.88}), (\ref{2.28}), (\ref{2.54}) and the simple fact:
		\begin{equation*}
			\int_{0}^{T}\max_{x \in \Omega}\theta_x^2dt\le \varepsilon \int_{0}^{T}\int_{ {\Omega}}\frac{\theta^{\beta}\theta_{xx}^2}{v}dxdt +C(\varepsilon) \int_{0}^{T}\int_{ {\Omega}}\theta_x^2 dxdt.
		\end{equation*}
		Combining (\ref{qwe}) and (\ref{1234}) gives
		\begin{equation*}
			\int_{0}^{T}\int_{{\Omega}}\theta_{xx}^2dxdt\le C,
		\end{equation*}
		which together with (\ref{2.95}) gives (\ref{2.94}).
		
		Next, it follows from (\ref{q1}) and (\ref{qwe}) that
		\begin{equation*}
			\int_{0}^{\infty}\left(\|\theta_x(,t)\|^2_{L^2(\Omega)} + \left|\frac{d}{dt}\|\theta_x(,t)\|^2_{L^2(\Omega)}\right|\right)dt\le C,
		\end{equation*}
		which directly gives
		\begin{equation}
			\lim_{t \rightarrow \infty}\|\theta_x(,t)\|_{L^2(\Omega)}=0.
			\label{2.99}
		\end{equation}
		
		Then, combining (\ref{1.1}) and (\ref{q1}) shows
		\begin{equation*}
			\begin{aligned}
				\int_{0}^{\infty}\left|\frac{d}{dt}\|v_x(,t)\|_{L^2(\Omega)}^2\right|dt &= 2\int_{0}^{\infty}\left|\int_{{\Omega}}u_{xx}v_xdx\right|dt \\
				& \le C\int_{0}^{\infty}\int_{{\Omega}}u_{xx}^2dxdt + C\int_{0}^{\infty}\int_{{\Omega}}v_x^2dxdt\\
				&\le C,
			\end{aligned}
		\end{equation*}
		which together with (\ref{q1}) implies
		\begin{equation}
			\lim_{t\rightarrow \infty}\|v_x(,t)\|_{L^2(\Omega)} = 0.
			\label{2.103}
		\end{equation}
		
		Finally, differentiating (\ref{1.2}) with respect to $t$ and multiplying the result by $u_t$, we have after integrating by parts and using (\ref{1.8}),
		\begin{equation}
			\begin{aligned}
				0 & = \frac{1}{2}\left(\int_{{\Omega}}u_t^2dx\right)_t + \int_{{\Omega}}\left(\frac{u_x-\theta}{v}\right)_tu_{xt}dx\\
				& = \frac{1}{2}\left(\int_{{\Omega}}u_t^2dx\right)_t + \int_{{\Omega}} \frac{u_{xt}^2}{v}dx - \int_{{\Omega}}\frac{u_x-\theta}{v^2}u_xu_{xt}dx - \int_{{\Omega}}\frac{\theta_tu_{xt}}{v}dx\\
				&\ge \frac{1}{2}\left(\int_{{\Omega}}u_t^2dx\right)_t + \frac{1}{2}\int_{{\Omega}}\frac{u_{xt}^2}{v}dx - C\int_{{\Omega}}(u_x^4+ u_x^2 + \theta_t^2)dx.
				\label{2.104}
			\end{aligned}
		\end{equation}
		
		Multiplying (\ref{2.104}) by $\zeta(t):=\min{\{1,t\}}$ and integrating the result over $(0,T)$ gives
		\begin{equation}
			\begin{aligned}
				&\sup_{0 \leq t \leq T}\zeta(t)\int_{{\Omega}}u_t^2dx + \int_{0}^{T}\zeta(t)\int_{{\Omega}}u_{xt}^2dxdt\\
				&\le C\int_{0}^{T}\int_{{\Omega}}u_t^2dxdt + C\int_{0}^{T}\int_{{\Omega}}(\theta_t^2 + u_x^2 + u_x^4)dxdt\\
				&\le C,
				\label{2.105}
			\end{aligned}
		\end{equation}
		where we have used $0\le \zeta' \le 1$, (\ref{q1}),(\ref{2.94}),(\ref{2.28}),(\ref{2.54}) and the simple fact:
		\begin{equation*}
			\begin{aligned}
				\int_{0}^{T}\int_{{\Omega}}u_x^4dxdt&\le \int_{0}^{T}\max_{x \in \Omega}u_x^2\int_{{\Omega}}u_x^2dxdt\\
				&\le C\int_{0}^{T}\int_{{\Omega}}u_{xx}^2dxdt+ C\int_{0}^{T}\int_{{\Omega}}u_x^2dxdt\le C.
			\end{aligned}
		\end{equation*}
		Combining (\ref{q10}) with (\ref{2.3}),  (\ref{2.28}), (\ref{2.54}), (\ref{2.59}), (\ref{q1}), (\ref{2.94}),  (\ref{2.105}) yields that for any $T\ge 1$
		\begin{equation*}
			\begin{aligned}
				\int_{T}^{\infty}\left|\frac{d}{dt}\|u_x(,t)\|_{L^2(\Omega)}^2\right|dt & \le \int_{T}^{\infty}\int_{{\Omega}}u_{xt}^2dxdt+ C\le C,
			\end{aligned}
		\end{equation*}
		which together with (\ref{2.28}) and (\ref{2.54}) gives
		\begin{equation*}
			\lim_{t \rightarrow \infty}\|u_x(,t)\|_{L^2(\Omega)} = 0.
		\end{equation*}
		Combining this, (\ref{2.99}) and (\ref{2.103}) gives (\ref{w1}), which together with (\ref{2.94}) finishes the proof of Lemma \ref{lemma 2.11}.
	\end{proof}

\end{document}